\documentclass[reqno]{amsart}

\usepackage{verbatim}
\usepackage{graphicx}
\usepackage{amssymb}
\usepackage{esint}

\newcommand{\R}{\mathbb{R}}

\newcommand{\grad}{\nabla}
\newcommand{\lap}{\Delta}

\theoremstyle{plain}
\newtheorem{thm}{Theorem}[section]
\newtheorem{lemma}[thm]{Lemma}
\newtheorem{defn}[thm]{Definition}

\numberwithin{equation}{section}

\newcommand{\beq}{\begin{eqnarray}}
\newcommand{\eeq}{\end{eqnarray}}
\newcommand{\beqno}{\begin{eqnarray*}}
\newcommand{\eeqno}{\end{eqnarray*}}

\allowdisplaybreaks

\begin{document}
\title[The Ericksen-Leslie system in $\R^2$]{Regularity and existence of global solutions to the Ericksen-Leslie system in $\R^2$} 

\author[J. Huang]{Jinrui Huang} 
\address{School of Mathematical Sciences\\
South China Normal University\\
Guangzhou 510631, China}
\email{huangjinrui1@163.com}
\author[F. Lin]{Fanghua Lin}
\address{Courant Institute of Mathematical Sciences\\
New York University\\
New York 100012, USA}
\email{linf@cims.nyu.edu}
\author[C. Wang]{Changyou Wang} 
\address{Department of Mathematics\\
University of Kentucky \\
Lexington, KY 40506, USA}

\email{cywang@ms.uky.edu} 
\date{May 17, 2013}

\begin{abstract}
In this paper, we first establish the regularity theorem for suitable weak solutions  to the Ericksen-Leslie system in $\R^2$.
Building on such a regularity, we then establish  the existence of a global weak solution to the Ericksen-Leslie system in $\R^2$ for any initial data in the energy space,
under the physical constraint conditions on the Leslie coefficients ensuring the dissipation of energy of the system, which is smooth away from at most finitely many
times.  This extends earlier works by Lin-Lin-Wang \cite{LLW}  on a simplified nematic liquid crystal flow.
\end{abstract}
\maketitle 








\setcounter{section}{0} \setcounter{equation}{0}
\section{Introduction}

The hydrodynamic theory of nematic liquid crystals was developed by Ericksen and Leslie during the period of 1958 through 1968 \cite{Ericksen1, Ericksen2,
Genn, Leslie1}. It is referred as the Ericksen-Leslie system in the literature. It reduces to the Ossen-Frank theory in the static case, which has been successfully
studied (see, e.g., Hardt-Lin-Kinderlehrer \cite{HLK}).

In $\R^3$, let $u=(u_1,u_2,u_3)^T:\R^3\to\R^3$ denote the fluid vector field of the underlying incompressible fluid,
and $d=(d_1,d_2,d_3)^T:\R^3\to \mathbb S^2$ denote the orientation order parameter representing the macroscopic average of nematic liquid crystal molecular directors.  Assume that the fluid is homogeneous (e.g., the fluid density $\rho\equiv 1$) and the inertial constant is zero (i.e., $\rho_1=0$), the general Ericksen-Leslie system in $\R^3$ consists of
the following equations (cf. \cite{Ericksen3, Leslie1, Leslie2, Lin-Liu3}):
\beq\label{ES1}
\begin{cases}
	\partial_t u+u\cdot \nabla u=\nabla\cdot\hat{\sigma},\\
	\nabla\cdot u=0,\\
	d\times \Big(g+\nabla\cdot(\frac{\partial W}{\partial(\nabla d)})-\frac{\partial W}{\partial d}\Big)=0,
\end{cases}
\eeq
where $\hat{\sigma}=(\hat{\sigma}_{ij})$ is the total stress given by
\beq
\hat{\sigma}_{ij}=-P\delta_{ij}-\frac{\partial W}{\partial d_{k,i}}d_{k,j}+\sigma_{ij}^L,
\eeq
with $P$ a scalar function representing the pressure, $W$ being the Ossen-Frank energy density function given by
\beq\label{OF}
W:=W(d,\nabla d)\equiv \frac12\Big(k_1(\nabla\cdot d)^2+k_2|d\times(\nabla\times d)|^2+k_3|d\cdot(\nabla\times d)|^2
+(k_2+k_4)[{\rm{tr}}(\nabla d)^2-(\nabla\cdot d)^2]\Big)
\eeq
for some elasticity constants $k_1, k_2, k_3, k_4$,
and $\sigma^L=\sigma^L(u,d)=(\sigma_{ij}^L(u,d))$ representing the Leslie stress tensor given by
\beq\label{leslie}
\sigma_{ij}^L(u,d)
=\mu_1 \sum_{k,p} d_k d_p A_{kp}  d_i d_j +\mu_2 N_i d_j +\mu_3 d_i N_j +\mu_4 A_{ij}+\mu_5\sum_{k} A_{ik} d_k d_j +\mu_6
\sum_{k} A_{jk}d_k d_i,
\eeq
for six viscous coefficients $\mu_1,\cdots,\mu_6$, called Leslie's coefficients, and
\beq\label{g-formula}
g_i=\lambda_1 N_i+\lambda_2 \sum_{j} d_j A_{ji}=\lambda_1\big(N_i+\frac{\lambda_2}{\lambda_1}\sum_{j} d_j A_{ji}\big).
\eeq
Throughout this paper, we use
\begin{eqnarray*}
A_{ij}=\frac12\left(\frac{\partial u_j}{\partial x_i}+\frac{\partial u_i}{\partial x_j}\right),
&&\Omega_{ij}=\frac12\left(\frac{\partial u_j}{\partial x_i}-\frac{\partial u_i}{\partial x_j}\right),\\
\omega_i=\dot{d_i}=\partial_td_i+(u\cdot\nabla) d_i, &&  N_{i}=\omega_i-\sum_{j}\Omega_{ij} d_j,
\end{eqnarray*}
denote the rate of strain tensor, the skew-symmetric part of the strain rate, the material derivative of $d$ and the rigid rotation part of the director
changing rate by fluid vorticity, respectively.

Due to the temperature dependence of the Leslie coefficients, various coefficients have different behavior: $\mu_4$ does not involve  the alignment
properties and hence is a smooth function of temperature, while all the other $\mu_i's$ describe couplings between molecule orientation and the flow,
and might be affected by the change of the nematic order parameter $d$. In this paper, we only consider the isothermal case in which all $\mu_i's$ are
assumed to be constants. The Leslie coefficients $\mu_i$'s and $\lambda_1$, $\lambda_2$ satisfy the following relations:
\begin{eqnarray}\label{necessary}
	\lambda_1=\mu_2-\mu_3, && \lambda_2=\mu_5-\mu_6,\\
	\mu_2+\mu_3&=&\mu_6-\mu_5.\label{parodi}
\end{eqnarray}
The relation (\ref{necessary}) is a necessary condition in order to satisfy the equation of motion identically, while (\ref{parodi}) is called Parodi's relation.
Under the assumption of Parodi's relation, the hydrodynamics of an incompressible nematic liquid crystal flow involves five independent Leslie's coefficients.
To avoid the complexity arising from the general Ossen-Frank energy functional, we consider the  elastically isotropic case $k_1=k_2=k_3=1$
and $k_4=0$ so that the Ossen-Frank energy reduces to the Dirichlet energy: $\displaystyle W(d,\nabla d)=\frac12|\nabla d|^2$. In this case, we have
$$
\frac{\partial W}{\partial d_{k,i}}d_{k,j}=\left\langle\frac{\partial d}{\partial x_i}, \frac{\partial d}{\partial x_j} \right\rangle=(\nabla d\odot\nabla d)_{ij},
\ \ \nabla\cdot\left(\frac{\partial{W}}{\partial(\nabla d)}\right)=\Delta d, \ \ \frac{\partial W}{\partial d}=0.
$$
As a consequence, the Ericksen-Leslie system (\ref{ES1}) can be written as
\beq\label{ES2}
\begin{cases}
\partial_t u+u\cdot \nabla u+\nabla P=-\nabla\cdot(\nabla d\odot\nabla d)+\nabla\cdot(\sigma^L(u,d)),\\
\nabla\cdot u = 0,\\
\partial_t d+u\cdot\nabla d-\Omega d+\frac{\lambda_2}{\lambda_1} Ad
=\frac{1}{|\lambda_1|}\left(\Delta d+|\nabla d|^2d\right)+\frac{\lambda_2}{\lambda_1}\left(d^TAd\right) d.
\end{cases}
\eeq
Since the general Ericksen-Leslie system is very complicated, earlier attempts of rigorous mathematical analysis of (\ref{ES2}) were made for a simplified
system that preserve the crucial energy dissipation feature as in (\ref{ES2}), pioneered by Lin \cite{Lin1} and Lin-Liu \cite{Lin-Liu1, Lin-Liu2}.
More precisely, by adding the penalty term $\displaystyle\frac{1}{4\epsilon^2}(1-|d|^2)^2$ ($\epsilon>0$) in the energy functional $W$
to remove the nonlinearities resulting from the nonlinear constraints $|d|=1$, Lin and Liu have studied in \cite{Lin-Liu1, Lin-Liu2} the following Ginzburg-Landau approximate system:
\beq\label{ES3}
\begin{cases}
\partial_t u+u\cdot\nabla u+\nabla P=\mu \Delta u-\nabla\cdot(\nabla d\odot\nabla d),\\
\nabla\cdot u = 0, \\
\partial_t d+u\cdot\nabla d=\Delta d+\frac{1}{\epsilon^2}\left(1-|d|^2\right)d.
\end{cases}
\eeq
They have established in \cite{Lin-Liu1} the existence of global weak solutions in dimensions $2$ and $3$, and global strong solutions of (\ref{ES3}) in dimension $2$, the local existence of strong solutions in dimension $3$, and  the existence of global strong solutions for large viscosity $\mu>0$ in dimension $3$. A partial regularity for suitable
weak solutions of (\ref{ES3}) in dimension $3$, analogous to Caffarelli-Kohn-Nirenberg \cite{CKN} on the Naiver-Stokes equation, has been proved in \cite{Lin-Liu2}. As already pointed out by \cite{Lin-Liu1}, it is still a challenging open problem as $\epsilon$ tends to zero,
whether solutions $(u^\epsilon, d^\epsilon)$ of (\ref{ES3}) converge to that of the following simplified Ericksen-Leslie system:
\beq\label{ES4}
\begin{cases}
\partial_t u+u\cdot\nabla u+\nabla P=\mu \Delta u-\nabla\cdot(\nabla d\odot\nabla d),\\
\nabla\cdot u = 0, \\
\partial_t d+u\cdot\nabla d=\Delta d+|\nabla d|^2d.
\end{cases}
\eeq
Very recently, there have been some important advances on (\ref{ES4}). For dimension $2$, Lin-Lin-Wang \cite{LLW} and Lin-Wang \cite{LW} have established the existence of a global unique weak solution of (\ref{ES4}) in any smooth bounded domain, under the initial and boundary conditions, which is smooth away from possibly finitely many times (see also Hong \cite{Hong}, Xu-Zhang \cite{XZ}, Hong-Xin \cite{HX},
and Lei-Li-Zhang\cite{LZZ} for related results
in $\R^2$). It is an open problem whether there exists a global weak solution of (\ref{ES4}) in dimension $3$.
There have been some partial results towards this problem. For example, the local existence and uniqueness of  strong solutions of (\ref{ES4}) has been proved by
Ding-Wen \cite{DW}, the blow-up criterion of local strong solutions of (\ref{ES4}), similar to Beale-Kato-Majda \cite{BKM} for Naiver-Stokes equations,
has been established by Huang-Wang \cite{HW}, the global well-posedness of (\ref{ES4}) for rough initial data $(u_0,d_0)$ with small ${\rm{BMO}}\times {\rm{BMO}}^{-1}$-norm has been shown by Wang \cite{W}, and the local well-posedness of (\ref{ES4}) for initial data $(u_0,d_0)$ with small
$L^3_{\rm{uloc}}(\R^3)$-norm of $(u_0,\nabla d_0)$ has been proved by Hineman-Wang \cite{Hine-Wang} (here $L^3_{\rm{uloc}}(\R^3)$ denotes
the locally uniform $L^3$-space on $\R^3$).

For the general Ericksen-Leslie system (\ref{ES2}) in $\R^3$, there have also been some recent works. For example, Lin-Liu \cite{Lin-Liu3} and Wu-Xu-Liu
\cite{WXL} have considered
its Ginzburg-Landau approximation:
\beq\label{ES5}
\begin{cases}
\partial_t u+u\cdot \nabla u+\nabla P=-\nabla\cdot(\nabla d\odot\nabla d)+\nabla\cdot(\sigma^L(u,d)),\\
\nabla\cdot u = 0,\\
\partial_t d+u\cdot\nabla d-\Omega d+\frac{\lambda_2}{\lambda_1} Ad
=\frac{1}{|\lambda_1|}\left(\Delta d+\frac{1}{\epsilon^2}(1-|d|^2)d\right),
\end{cases}
\eeq
and have established the existence of global weak solutions, and the local existence and uniqueness of strong solutions of (\ref{ES5}) under
certain conditions on the Leslie coefficients $\mu_i$'s. In particular, there have been results developed by \cite{WXL} concerning
the role of Parodi's condition (\ref{parodi}) in the well-posedness and stability of (\ref{ES5}).  Most recently, Wang-Zhang-Zhang \cite{WZZ}
have studied the general Ericksen-Leslie system (\ref{ES2}) and established the local well-posedness,  and the global well-posedness
for small initial data under a seemingly optimal condition on the Leslie coefficients $\mu_i$'s.

In this paper, we are mainly interested in both the regularity and existence of global weak solutions of the initial value problem of the
general Ericksen-Leslie system  (\ref{ES2}) in $\R^2$.

In $\R^2$,  however, we need to modify several terms inside the system (\ref{ES2}) in order to make it into a closed system.
Since $u$ is a planar vector field in $\R^2$, both $\Omega$ and $A$ are horizontal $2\times 2$-matrices, henceforth
we assume that
\begin{eqnarray} \label{reduction1}
\Omega d:=(\Omega\hat{d}, 0)^T, \  Ad:=(A\hat{d}, 0)^T, \ N:=\left(\partial_t d+u\cdot\nabla d\right)^T-\big(\Omega \hat{d}, 0\big)^T,
\end{eqnarray}
as vectors in $\R^3$, while $\displaystyle\hat{N}:=(\partial_t \hat d+u\cdot\nabla \hat d-\Omega \hat{d})^T$ is
a vector in $\R^2$.
Here
$$\hat{d}=(d_1,d_2,0)^T \ {\rm{for}}\ d=(d_1,d_2,d_3)^T\in\R^3,$$
and $\sigma^L(u,d)$ is a $2\times 2$-matrix valued function given by
\beq\label{reduction2}
\sigma_{ij}^L(u,d):=\mu_1\sum_{k,p=1}^2 d_kd_pA_{kp} d_i d_j+\mu_2 N_i d_j +\mu_3 N_j d_i +\mu_4 A_{ij}
+\mu_5 \sum_{k=1}^2 A_{ik} d_k d_j +\mu_6 \sum_{k=1}^2 A_{jk} d_k d_i
\eeq
for $1\le i, j\le 2$, and
$$d^TAd:=\hat{d}^T A \hat{d} \equiv\sum_{i,j=1}^2 A_{ij} d_i d_j.$$

We will consider the initial value problem of (\ref{ES2}) in $\R^2$, i.e.,
\beq\label{IV}
(u,d)\Big|_{t=0}=(u_0, d_0) \ {\rm{in}}\ \R^2
\eeq
for any given $u_0\in \mathbf H$, and $d_0\in {H}^1_{e_0}(\R^2,\mathbb S^2)$. Here we denote the relevant function spaces
$$\mathbf {H}={\rm the}\ {\rm closure}\ {\rm of}\ C_0^\infty(\mathbb{R}^2,\mathbb{R}^2)\cap \{v\ |\ \nabla\cdot v=0\}\ {\rm in}\ L^2(\mathbb{R}^2,\mathbb{R}^2),$$
$$H^k_{e_0}(\R^2,\mathbb S^2)
=\Big\{d:\mathbb R^2\to \mathbb S^2\ | \ \ d-e_0\in H^k(\R^2,\mathbb R^3)\Big\} \ (k\in \mathbb N_+)$$
for some constant vector $e_0\in\mathbb S^2$,
and
$$\mathbf {J}={\rm the}\ {\rm closure}\ {\rm of}\ C_0^\infty(\mathbb{R}^2,\mathbb{R}^2)\cap \{v \ | \ \nabla\cdot v=0\}
\ {\rm in}\ H^1(\mathbb{R}^2,\mathbb{R}^2).$$

\begin{defn} For $0<T\le\infty$, $u\in  L^2([0,T],\mathbf{H})$ and $d\in L^2\left([0,T],H^1_{e_0}(\mathbb{R}^2,\mathbb S^2)\right)$
is called a weak solution of the Ericksen-Leslie system (\ref{ES2}) together with the initial condition (\ref{IV}) in $\R^2$, if
\beqno
      &&-\int_0^T\int_{\mathbb{R}^2}\langle u,\ \psi^\prime\phi\rangle-\int_0^T\int_{\mathbb{R}^2}u\otimes u:\psi\nabla\phi
+\int_0^T\int_{\mathbb{R}^2}\sigma^L(u,d):\psi\nabla\phi
      \\&=&
      -\psi(0)\int_{\mathbb{R}^2}\langle u_0,\ \phi\rangle+\int_0^T\int_{\mathbb{R}^2}\langle\nabla d\odot\nabla d,\ \psi\nabla\phi\rangle,
\eeqno
and
\beqno
      &&-\int_0^T\int_{\mathbb{R}^2}\langle d,\ \psi^\prime\tilde{\phi}\rangle+\int_0^T\int_{\mathbb{R}^2}\left\langle u\cdot\nabla d-\Omega d+\frac{\lambda_2}{\lambda_1}Ad,\ \psi\tilde{\phi}\right\rangle
      \\&=&
      -\psi(0)\int_{\mathbb{R}^2}\langle d_0,\ \tilde{\phi}\rangle+\int_0^T\int_{\mathbb{R}^2}\Big[-\frac{1}{|\lambda_1|}\langle\nabla d, \ \psi\nabla\tilde{\phi}\rangle
+\frac{1}{|\lambda_1|}|\nabla d|^2\langle d,\ \psi\tilde{\phi}\rangle+\frac{\lambda_2}{\lambda_1} (d^TAd)\langle d,\ \psi\tilde{\phi}\rangle\Big].
\eeqno
for any $\psi\in C^\infty([0,T])$ with $\psi(T)=0$, $\phi\in\mathbf J$, and $\tilde{\phi}\in H^1(\mathbb{R}^2,\R^3)$.\footnote{By using (\ref{ES2})$_3$, we see that
$N\in L^2_tH^{-1}_x+L^1_tL^1_x$. Hence $\sigma^L(u,d)\in L^1_t (H^{-1}_x\cdot H^1_x)
+L^1_t(L^1_x\cdot L^\infty_x)$ and $\displaystyle\int_0^T\int_{\R^2}\sigma^{L}(u,d):\psi\nabla\phi$ is well defined.}
\end{defn}

In this paper, we will establish the regularity of {\it suitable} weak solutions of (\ref{ES2}),  and the existence and uniqueness of global weak solutions to (\ref{ES2}) and (\ref{IV}) in $\R^2$.
As consequences, these extend the previous works by Lin-Lin-Wang \cite{LLW} to the general case.

For $x_0\in\mathbb{R}^2$, $t_0\in (0,+\infty)$, $z_0=(x_0,t_0)$ and $0<r\le\sqrt{t_0}$, denote
$$B_r(x_0)=\{x\in \mathbb{R}^2\ | \  |x-x_0|\leq r\},\ \ P_r(z_0)=B_r(x_0)\times[t_0-r^2,t_0].$$
When $x_0=(0,0)$ and $t_0=0$, we simply denote $B_r=B_r(0),\ P_r=P_r(0).$

Now we introduce the notion of suitable weak solutions of (\ref{ES2}).
\begin{defn} For $0<T<+\infty$ and a domain $O\subset\R^2$, a weak solution $u\in L^2(\R^2\times [0,T],\R^2)$,
with $\nabla\cdot u=0$, and $d\in L^2_tH^1_x(O\times [0, T],\mathbb S^2)$ of (\ref{ES2}) is called a {\rm{suitable}} weak solution of (\ref{ES2})
if, in addition, $u\in \big(L^\infty_tL^2_x\cap L^2_t H^1_x\big)(O\times [0,T],\R^2)$, 
$d\in \big(L^\infty_t H^1_x\cap L^2_t H^2_x\big)(O\times [0,T],\mathbb S^2)$, and $P\in L^2(O\times [0,T])$.
\end{defn}

\begin{thm}\label{regularity} For $0<T<+\infty$ and a domain $O\subset\R^2$, assume that $u\in \big(L^\infty_tL^2_x\cap L^2_t H^1_x\big)(O\times [0,T],\R^2)$,
$d\in L^\infty ([0,T], H^1(O,\mathbb S^2))\cap L^2 ([0,T], H^2(O,\mathbb S^2))$, and $P\in L^2(O\times [0,T])$ is a suitable weak solution of (\ref{ES2}).
Assume both (\ref{necessary}) and (\ref{parodi}) hold. If, in additions, the Leslie coefficients $\mu_i$'s satisfy
\beq\label{Leslie_condition}
\lambda_1<0, \ \mu_1-\frac{\lambda_2^2}{\lambda_1}\ge 0, \ \mu_4>0, \ \mu_5+\mu_6\ge -\frac{\lambda_2^2}{\lambda_1},
\eeq
then $(u,d)\in C^\infty(O\times (0,T],\R^2\times \mathbb S^2)$.
\end{thm}

Employing the priori estimate given by the proof of Theorem \ref{regularity}, we will prove the existence of global weak solutions of (\ref{ES2}) and (\ref{IV})
that enjoy partial smoothness properties.

\begin{thm}\label{existence} For any $u_0\in \mathbf H$ and $d_0\in H^1_{e_0}(\R^2, \mathbb S^2)$, assume the conditions (\ref{necessary}),
(\ref{parodi}), and (\ref{Leslie_condition}) hold. Then there is a global weak solution $u\in L^\infty([0,+\infty), \mathbf H)\cap L^2([0,+\infty),
\mathbf J)$ and $d\in L^\infty([0,+\infty), H^1_{e_0}(\R^2,\mathbb S^2))$ of the general Ericksen-Leslie system (\ref{ES2}) and
(\ref{IV}) such that the following properties hold:\\
(i) There exist a nonnegative integer $L$, depending only on $(u_0,d_0)$, and $0<T_1<\cdots<T_L<+\infty$ such that
$$
(u,d)\in C^\infty\big(\R^2\times ((0,+\infty)\setminus\{T_i\}_{i=1}^L), \R^2\times \mathbb S^2\big).
$$
(ii) Each singular time $T_i$, $1\le i\le L$, can be characterized by
\beq\label{energy_concentration}
\liminf_{t\uparrow T_i}\max_{x\in\R^2}\int_{B_r(x)}(|u|^2+|\nabla d|^2)(y,t)\ge 8\pi, \ \forall r>0.
\eeq
Moreover, there exist $x_m^i\rightarrow x_0^i\in \R^2$, $t_m^i\uparrow T_i$, $r_m^i\downarrow 0$ and a non-trivial smooth harmonic map
$\omega_i:\R^2\to \mathbb S^2$ with finite energy such that as $m\rightarrow\infty$,
$$(u_m^i, d_m^i)\rightarrow (0,\omega_i) \ {\rm{in}}\ C^2_{\rm{loc}}(\R^2\times [-\infty, 0]),$$
where
$$u_m^i(x,t)=r_m^iu(x_m^i+r_m^i x, t_m^i+(r_m^i)^2 t),
\  d_m^i(x,t)=d(x_m^i+r_m^i x, t_m^i+(r_m^i)^2 t).$$
(iii) Set $T_0=0$. Then
$$(d_t, \nabla^2 d)\in \bigcap_{j=0}^{L-1}\bigcap_{\epsilon>0}L^2(\R^2\times [T_j, T_{j+1}-\epsilon]) 
\bigcap\big(\bigcap_{T_L<T<+\infty} L^2(\R^2\times [T_L, T])\big).$$
(iv) There exist $t_k\uparrow +\infty$ and a smooth harmonic map $d_\infty\in C^\infty(\R^2,\mathbb S^2)$ with finite energy such that
$u(\cdot, t_k)\rightarrow 0$ in $H^1(\R^2)$ and $d(\cdot, t_k)\rightharpoonup d_\infty$ in $H^1(\R^2)$, and there exist a nonnegative
integer $l$, $\{x_i\}_{i=1}^l\subset\R^2$, and nonnegative integers $\{m_i\}_{i=1}^l$ such that
\beq\label{energy_id}
|\nabla d(\cdot, t_k)|^2\,dx\rightharpoonup |\nabla d_\infty|^2\,dx+8\pi\sum_{i=1}^l m_i\delta_{x_i}.
\eeq
(v) If either the third component of $d_0$,  $(d_0)_3$, is nonnegative, or
$$\int_{\R^2}(|u_0|^2+|\nabla d_0|^2)\le 8\pi,$$
then $(u,d)\in C^\infty\left(\R^2\times (0,+\infty),\R^2\times\mathbb S^2\right)$. Moreover, there exist $t_k\uparrow +\infty$ and a smooth harmonic
map $d_\infty\in C^\infty(\R^2, \mathbb S^2)$ with finite energy such that
$$(u(\cdot, t_k), d(\cdot, t_k))\rightarrow (0, d_\infty) \ {\rm{in}}\ C^2_{\rm{loc}}(\R^2).$$
\end{thm}

An important first step to prove Theorem \ref{regularity} is to establish the decay lemma \ref{le:epsilon} under the small energy condition, which is proved by a blow-up argument.
Here the local energy inequality (\ref{local_energy_ineq1}) for suitable weak solutions to (\ref{ES2}) plays a very important role, which depends
on the conditions (\ref{necessary}), (\ref{parodi}), and (\ref{Leslie_condition}) heavily.
In contrast with earlier arguments developed by \cite{LLW} on the simplified nematic liquid crystal equation (\ref{ES4}),
where the limiting equation resulting from the blow up process is the linear Stokes equation and the linear heat equation,
the new linear system (\ref{lq3}) arising from the blow-up process of the general Ericksen-Leslie system (\ref{ES2}) is a coupling system.
It is an interesting question to establish its smoothness. The proof of regularity of (\ref{lq3}) is based on higher order local energy inequalities.
The cancelation properties among the coupling terms play critical roles in the argument of various local or global energy inequalities
 for both the linear system (\ref{lq3}) and the nonlinear system
(\ref{ES2}). The second step is to establish a higher integrability estimate of suitable weak solutions to (\ref{ES2}) under the small energy condition, which
is done by employing the techniques of Riesz potential estimates between parabolic Morrey spaces developed by \cite{HW} and \cite{Hine-Wang}. The third
step is to establish an arbitrary higher order energy estimate of (\ref{ES2}) under the small energy condition.  With the regularity theorem \ref{regularity}, we show
the existence theorem \ref{existence} by adapting the scheme developed by \cite{LLW}.

Motivated by the uniqueness theorem
proved by \cite{LLW} and \cite{XZ} on (\ref{ES4}), we believe that the weak solution obtained in Theorem \ref{existence}
is also unique in its own class and plan to address it in a forthcoming article.

The paper is organized as follows. In section two, we derive both local and global energy inequality for suitable weak solutions of (\ref{ES2}).
In section three, we prove an $\epsilon$-regularity theorem for (\ref{ES2}) first and then prove Theorem \ref{regularity}. In section four, we prove
the existence theorem \ref{existence}.

\setcounter{section}{1} \setcounter{equation}{1}
\section{Global and local energy inequalities of Ericksen-Leslie' system in $\R^2$}

In this section, we will establish both global and local energy inequality for suitable weak solutions of (\ref{ES2}) in $\R^2$ under the conditions
(\ref{necessary}), (\ref{parodi}), and (\ref{Leslie_condition}). We begin with the global energy inequality.

\begin{lemma}\label{global_energy_ineq} For $0<T\le +\infty$, assume the conditions (\ref{necessary}), (\ref{parodi}), and (\ref{Leslie_condition}) hold.
If $u\in L^\infty_tL^2_x\cap L^2_tH^1_x(\R^2\times [0,T], \R^2)$, $d\in L^\infty([0,T], H^1_{e_0}(\R^2,\mathbb S^2))
\cap L^2([0,T], H^2_{e_0}(\R^2, \mathbb S^2))$, and $P\in L^2(\R^2\times [0,T])$ is a suitable weak solution of
 the Ericksen-Leslie system (\ref{ES2}). Then for any $0\le t_1<t_2\le T$, it holds
\beq\label{global_energy_ineq1}
\int_{\R^2}(|u|^2+|\nabla d|^2)(t_2)+\int_{t_1}^{t_2}\int_{\R^2}\Big[\mu_4|\nabla u|^2+\frac2{|\lambda_1|}
|\Delta d+|\nabla d|^2 d|^2\Big]
\le \int_{\R^2}(|u|^2+|\nabla d|^2)(t_1).
\eeq
\end{lemma}
\begin{proof} Since $|d|=1$, we have $|\nabla d|^2+\langle d,\Delta d\rangle=0$. This, combined with the fact
$d\in  L^2([0,T], H^2_{e_0}(\R^2, \mathbb S^2))$, implies that $\nabla d\in L^4(\R^2\times [0,T])$. Since $u\in L^\infty_tL^2_x\cap L^2_tH^1_x(\R^2\times [0,T], \R^2)$,
it follows from Ladyzhenskaya's inequality that $u\in L^4(\R^2\times [0,T])$.

For $\eta\in C_0^\infty(\R^2)$ or $\eta\equiv 1$ in $\R^2$, multiplying (\ref{ES2})$_1$ by $u\eta^2$, integrating the resulting equation over $\R^2$, and using $\nabla\cdot u=0$,  we obtain
\begin{eqnarray}\label{sigma_est1}
\frac{d}{dt}\int_{\R^2}|u|^2\eta^2&=& 2\int_{\R^2} \eta^2\left[\nabla d\odot\nabla d:\nabla u
-\sigma^L(u,d):\nabla u\right]\nonumber\\
&&+\int_{\R^2}\left[(|u|^2+2P) u\cdot\nabla(\eta^2)+2\left(\nabla d\odot\nabla d-\sigma^L(u,d)\right):u\otimes\nabla (\eta^2)\right].
\end{eqnarray}
Using (\ref{necessary}), (\ref{parodi}), the symmetry of $A$, and the skew-symmetry of $\Omega$, we find
\begin{eqnarray}\label{sigma_est2}
&&\int_{\R^2}\eta^2\sigma^L(u,d):\nabla u\nonumber\\
&=&\int_{\R^2}\eta^2\Big[\mu_1 \hat{d}_k\hat{d}_pA_{kp}\hat{d}_i\hat{d}_j +\mu_2\hat{N}_i \hat{d}_j+\mu_3 \hat{N}_j \hat{d}_i+\mu_4 A_{ij}
+\mu_5 A_{ik}\hat{d}_k \hat{d}_j
+\mu_6 A_{jk}\hat{d}_k \hat{d}_i\Big](A_{ij}+\Omega_{ij})\nonumber\\
&=&\int_{\R^2}\eta^2\Big[\mu_1|A:{\hat d}\otimes {\hat d}|^2+\mu_4 |A|^2+(\mu_2+\mu_3)\hat{d}\cdot(A\cdot\hat{N})
+(\mu_2-\mu_3)\hat{d}\cdot(\Omega\cdot\hat{N})\nonumber\\
&&\qquad+ (\mu_5+\mu_6)|A\cdot \hat d|^2+(\mu_5-\mu_6)(A\cdot\hat{d})(\Omega\cdot\hat{d})\Big]\nonumber\\
&=&\int_{\R^2}\eta^2\Big[\mu_1|A:{\hat d}\otimes {\hat d}|^2+\mu_4 |A|^2+ (\mu_5+\mu_6)|A\cdot \hat d|^2\nonumber\\
&&\qquad+\lambda_1\hat{N}\cdot(\Omega\cdot\hat{d})-\lambda_2 \hat{N}\cdot(A\cdot{\hat d})+\lambda_2 (A\cdot\hat{d})(\Omega\cdot\hat{d})\Big].
\end{eqnarray}
Putting (\ref{sigma_est2}) into (\ref{sigma_est1}), we have
\begin{eqnarray}
\frac{d}{dt}\int_{\R^2}|u|^2\eta^2&=& 2\int_{\R^2} \eta^2\Big[\nabla d\odot\nabla d:\nabla u
-\mu_1|A:{\hat d}\otimes {\hat d}|^2-\mu_4 |A|^2- (\mu_5+\mu_6)|A\cdot \hat d|^2\nonumber\\
&&\qquad-\lambda_1\hat{N}\cdot(\Omega\cdot\hat{d})+\lambda_2 \hat{N}\cdot(A\cdot{\hat d})-\lambda_2 (A\cdot\hat{d})(\Omega\cdot\hat{d})\Big]\nonumber\\
&&+\int_{\R^2}\left[(|u|^2+2P) u\cdot\nabla(\eta^2)+2\left(\nabla d\odot\nabla d-\sigma^L(u,d)\right):u\otimes\nabla (\eta^2)\right].\label{sigma_est3}
\end{eqnarray}
Since $(\Delta d+|\nabla d|^2 d)\in L^2(\R^2\times [0,T])$, it follows from the equation $(\ref {ES2})_3$ that $(\partial_t d+u\cdot\nabla d)\in
L^2(\R^2\times [0,T])$. Multiplying $(\ref{ES2})_3$ by $\eta^2\Delta d$ and integrating the resulting equation
over $\R^2$ yields that
\beq
    \label{d local}
    &&\frac{d}{dt}\int_{B_1}\frac12|\nabla d|^2\eta^2+\int_{B_1}\frac{1}{|\lambda_1|}|\Delta d+|\nabla d|^2d|^2\eta^2\nonumber\\
&=&-\int_{\R^2}\langle\partial_t d, \nabla d\rangle\cdot\nabla(\eta^2)
+\int_{\R^2}\eta^2\left[\langle u\cdot\nabla d, \Delta d\rangle+\langle \frac{\lambda_2}{\lambda_1}A\hat d
-\Omega\hat d, \Delta \hat d\rangle +\frac{\lambda_2}{\lambda_1}(\hat{d}^TA\hat{d})|\nabla d|^2\right].
\eeq
Using integration by parts and $\nabla\cdot u=0$, we see
\beq
\int_{\R^2}\eta^2\langle u\cdot\nabla d, \Delta d\rangle
=-\int_{\R^2}\eta^2\nabla d\odot\nabla d:\nabla u+\int_{\R^2}\left[\frac12|\nabla d|^2 u\cdot\nabla(\eta^2)
-\nabla d\odot\nabla d: u\otimes \nabla(\eta^2)\right].\nonumber
\eeq
Substituting this into (\ref{d local}), we obtain
\beq\label{d_local1}
 &&\frac{d}{dt}\int_{B_1}\frac12|\nabla d|^2\eta^2+\int_{B_1}\frac{1}{|\lambda_1|}|\Delta d+|\nabla d|^2d|^2\eta^2\nonumber\\
&=&\int_{\R^2}\eta^2\left[\langle \frac{\lambda_2}{\lambda_1}A\hat d
-\Omega\hat d, \Delta \hat d\rangle +\frac{\lambda_2}{\lambda_1}(\hat{d}^TA\hat{d})|\nabla d|^2
-\eta^2\nabla d\odot\nabla d:\nabla u\right]\nonumber\\
&+&\int_{\R^2}\left[\frac12|\nabla d|^2 u\cdot\nabla(\eta^2)
-\nabla d\odot\nabla d: u\otimes \nabla(\eta^2)-\langle\partial_t d, \nabla d\rangle\cdot\nabla(\eta^2)\right].
\eeq
Adding (\ref{sigma_est3}) together with (\ref{d_local1}), we obtain
\begin{eqnarray}\label{u-d-estimate1}
&&\frac{d}{dt}\int_{\R^2}(|u|^2+|\nabla d|^2)\eta^2
+2\int_{\R^2}\left[\mu_4|A|^2+\frac{1}{|\lambda_1|}|\Delta d+|\nabla d|^2d|^2\right]\eta^2\nonumber\\
&=&-2\int_{\R^2} \eta^2\Big[\mu_1|A:{\hat d}\otimes {\hat d}|^2 +(\mu_5+\mu_6)|A\cdot \hat d|^2
+\lambda_1\hat{N}\cdot(\Omega\cdot\hat{d})-\lambda_2 \hat{N}\cdot(A\cdot{\hat d})\nonumber\\
&&\qquad\ \ \ \ \ \ \ +\lambda_2 (A\cdot\hat{d})(\Omega\cdot\hat{d})
-\langle \frac{\lambda_2}{\lambda_1}A\hat d
-\Omega\hat d, \Delta \hat d\rangle -\frac{\lambda_2}{\lambda_1}(\hat{d}^TA\hat{d})|\nabla d|^2\Big]
\nonumber\\
&&+\int_{\R^2}\left[(|u|^2+2P) u\cdot\nabla(\eta^2)+2\left(\nabla d\odot\nabla d-\sigma^L(u,d)\right):u\otimes\nabla (\eta^2)\right]\nonumber\\
&&+\int_{\R^2}\left[|\nabla d|^2 u\cdot\nabla(\eta^2)
-2\nabla d\odot\nabla d: u\otimes \nabla(\eta^2)-2\langle\partial_t d, \nabla d\rangle\cdot\nabla(\eta^2)\right].
\end{eqnarray}
Denote the first term in the right hand side of (\ref{u-d-estimate1}) as {\bf I}.
In order to estimate {\bf I}, we need to use (\ref{ES2})$_3$ to make crucial cancelations among terms of
{\bf I} as follows.
\beq\label{cancel1}
\lambda_1\hat{N}\cdot(\Omega\cdot\hat{d})&=&\lambda_1\hat{N}_i\Omega_{ij}\hat{d_j}\nonumber\\
&=&\left[-\lambda_2A_{ik}\hat{d}_k-\Delta \hat d_i+\lambda_2(\hat{d}^TA\hat d)\hat d_i)\right]\Omega_{ij}\hat d_j
\nonumber\\
&=&\left[-\lambda_2A_{ik}\hat{d}_k-\Delta \hat d_i\right]\Omega_{ij}\hat d_j
=-\lambda_2 (A\cdot\hat d)(\Omega\cdot\hat d)-\langle \Omega \hat d, \Delta\hat d\rangle ,
\eeq
while
\beq\label{cancel2}
-\lambda_2 \hat{N}\cdot(A\cdot{\hat d})
&=&-\lambda_2 \hat{N}_i A_{ij} \hat{d}_j\nonumber\\
&=& -\lambda_2\left[-\frac{\lambda_2}{\lambda_1}A_{ij}\hat d_j-\frac{1}{\lambda_1}(\Delta \hat d_i+|\nabla d|^2\hat d_i)+\frac{\lambda_2}{\lambda_1}(\hat d^T A \hat d)\hat d_i\right]A_{ij}\hat d_j\nonumber\\
&=&\frac{\lambda_2^2}{\lambda_1}|A\cdot\hat d|^2-\frac{\lambda_2^2}{\lambda_1}|\hat d^T Ad|^2
+\frac{\lambda_2}{\lambda_1}\langle A\hat d,\Delta \hat d\rangle +\frac{\lambda_2}{\lambda_1}(\hat d^T A \hat d)|\nabla d|^2.
\eeq
Since $A:\hat d\otimes \hat d=\hat d^T A\hat d$, we have, by substituting (\ref{cancel1}) and (\ref{cancel2}) into
{\bf I},
that
\beq\label{I-estimate}
{\mathbf{I}}=-2\int_{\R^2}\Big[\left(\mu_1-\frac{\lambda_2^2}{\lambda_1}\right)|A:\hat d\otimes\hat d|^2
+\left(\mu_5+\mu_6+\frac{\lambda^2_2}{\lambda_1}\right)|A\cdot\hat d|^2\Big].
\eeq
Substituting (\ref{I-estimate}) into (\ref{u-d-estimate1}), we finally obtain
\begin{eqnarray}\label{u-d-estimate2}
&&\frac{d}{dt}\int_{\R^2}(|u|^2+|\nabla d|^2)\eta^2
+\int_{\R^2}\left[\mu_4|\nabla u|^2+\frac{2}{|\lambda_1|}|\Delta d+|\nabla d|^2d|^2\right]\eta^2\nonumber\\
&=&-2\int_{\R^2}\eta^2\Big[\left(\mu_1-\frac{\lambda_2^2}{\lambda_1}\right)|A:\hat d\otimes\hat d|^2
+\left(\mu_5+\mu_6+\frac{\lambda^2_2}{\lambda_1}\right)|A\cdot\hat d|^2\Big]
\nonumber\\
&&+\int_{\R^2}\left[(|u|^2+2P) u\cdot\nabla(\eta^2)+2\left(\nabla d\odot\nabla d-\sigma^L(u,d)\right):u\otimes\nabla (\eta^2)\right]\nonumber\\
&&+\int_{\R^2}\left[|\nabla d|^2 u\cdot\nabla(\eta^2)
-2\nabla d\odot\nabla d: u\otimes \nabla(\eta^2)-2\langle\partial_t d, \nabla d\rangle\cdot\nabla(\eta^2)+\mu_4
\langle u\cdot\nabla u,\nabla(\eta^2)\rangle \right].
\end{eqnarray}
Here we have used the fact $\nabla\cdot u=0$ and the following identity:
\beq\label{u-A-identity}
\int_{\R^2} |A|^2\eta^2=\frac12\int_{\R^2}|\nabla u|^2\eta^2-\frac12\int_{\R^2}\langle (u\cdot\nabla) u, \nabla(\eta^2)
\rangle.
\eeq
If $\eta\equiv 1$, then (\ref{u-d-estimate2}) and the condition (\ref{Leslie_condition}) imply
\beq\label{global_energy_ineq2}
&&\frac{d}{dt}\int_{\R^2}(|u|^2+|\nabla d|^2)
+\int_{\R^2}\left[\mu_4|\nabla u|^2+\frac{2}{|\lambda_1|}|\Delta d+|\nabla d|^2d|^2\right]\nonumber\\
&=&-2\int_{\R^2}\Big[\left(\mu_1-\frac{\lambda_2^2}{\lambda_1}\right)|A:\hat d\otimes\hat d|^2
+\left(\mu_5+\mu_6+\frac{\lambda^2_2}{\lambda_1}\right)|A\cdot\hat d|^2\Big]
\le 0.
\eeq
Integrating (\ref{global_energy_ineq2}) over $0\le t_1\le t_2\le T$ yields
(\ref{global_energy_ineq1}).
This completes the proof of lemma \ref{global_energy_ineq}.
\end{proof}

We also need the following local energy inequality in the proofs of our main theorems.
\begin{lemma}\label{local_energy_ineq}
 For $0<T\le +\infty$, assume the conditions (\ref{necessary}), (\ref{parodi}), and (\ref{Leslie_condition}) hold.
If $u\in L^\infty_tL^2_x\cap L^2_tH^1_x(\R^2\times [0,T], \R^2)$, $d\in L^\infty([0,T], H^1_{e_0}(\R^2,\mathbb S^2))
\cap L^2([0,T], H^2_{e_0}(\R^2, \mathbb S^2))$, and $P\in L^2(\R^2\times [0,T])$ is a suitable weak solution
of the Ericksen-Leslie system (\ref{ES2}). Then for any $0\le t_1<t_2\le T$ and $\eta\in C_0^\infty(\R^2)$, it holds
\beq\label{local_energy_ineq1}
&&\int_{\R^2}\eta^2(|u|^2+|\nabla d|^2)(t_2)+\int_{t_1}^{t_2}\int_{\R^2}\eta^2\Big[\mu_4|\nabla u|^2+\frac2{|\lambda_1|}
|\Delta d+|\nabla d|^2 d|^2\Big]\nonumber\\
&\le& \int_{\R^2}\eta^2(|u|^2+|\nabla d|^2)(t_1)
+C\int_{t_1}^{t_2}\int_{\R^2}\Big[(|u|^2+|\nabla u|+|u||\nabla d|+|\nabla d|^2+|\nabla^2 d|+|P|)|u||\nabla(\eta^2)|\nonumber\\
&&\qquad+(|\nabla u|+|u||\nabla d|+|\nabla d|^2+|\nabla^2 d|)|\nabla d||\nabla(\eta^2)|\Big].
\eeq
\end{lemma}
\begin{proof} It suffices to estimate the last two terms in the right hand of (\ref{u-d-estimate2}).
Denote these two terms by {\bf II} and {\bf III}.  To do so, first observe
that by (\ref{ES2})$_3$ it holds that
$$|\partial_t d|\le C(|u||\nabla d|+|\nabla u|+|\nabla^2 d|+|\nabla d|^2),$$
and hence
$$|\sigma^L(u,d)|\le C(|A|+|N|)\le C(|\nabla u|+|\partial_t d|+|u||\nabla d|)
\le C(|\nabla u|+|u||\nabla d|+|\nabla^2 d|+|\nabla d|^2).$$
With these estimates, we can show that
\beq\nonumber
|{\bf II}|&\le&
C\int_{\R^2}(|u|^2+|P|+|\nabla d|^2+|\sigma^L(u,d)|)|u||\nabla(\eta^2)|\\
&\le& C\int_{\R^2}(|u|^2+|P|+|\nabla d|^2+|u||\nabla d|+|\nabla u|+|\nabla^2 d|)|u||\nabla(\eta^2)|,
\nonumber
\eeq
and
\beq\nonumber
|{\bf III}|&\le&
C\int_{\R^2}(|\nabla d|^2|u|+|\partial_t d||\nabla d|+|u||\nabla u|)|\nabla(\eta^2)|\\
&\le& C\int_{\R^2}\left(|\nabla d|^2|u|+|u||\nabla u|+(|\nabla u|+|u||\nabla d|+|\nabla d|^2+|\nabla^2 d|)|\nabla d|\right)
|\nabla(\eta^2)|.
\nonumber
\eeq
Putting these estimates of {\bf II} and {\bf III} into (\ref{u-d-estimate2}) yields (\ref{local_energy_ineq1}).
This completes the proof.
\end{proof}

\setcounter{section}{2} \setcounter{equation}{2}
\section{$\epsilon$-regularity of the Ericksen-Leslie system (\ref{ES2}) in $\R^2$}

In this section, we will establish the regularity of suitable weak solutions of the Ericksen-Leslie system (\ref{ES2}) in $\R^2$, under a smallness condition.
The crucial step is the following decay lemma under the smallness condition.

\begin{lemma}\label{le:epsilon} Assume that the conditions (\ref{necessary}),
(\ref{parodi}), and (\ref{Leslie_condition}) hold. There exist $\epsilon_0>0$ and $\theta_0\in (0, \frac12)$ such that  for $0<T<+\infty$ and a bounded domain $O\subset\R^2$,
if $u\in L^\infty_t L^2_x\cap L^2_t H^1_x(O\times [0, T], \R^2)$,
$d\in L^\infty([0,T],  H^1(O,\mathbb S^2))\cap L^2([0,T],  H^2(O, \mathbb S^2))$, and $P\in L^2(O\times [0,T])$
is a suitable weak solution of the Ericksen-Leslie system (\ref{ES2}) in $O\times [0,T]$,  which satisfies, for some $z_0=(x_0,t_0)\in O\times (0,T]$ and
$0<r_0\le \min\{d(x_0,\partial O), \sqrt{t_0}\}$,
\beq
    \label{small_condition}
    \nonumber
    \Phi(u,d,P,z_0, r)\leq \epsilon_0,
\eeq
then
\beq\label{decay_ineq}
    \Phi(u,d,P, z_0, \theta_0r)\leq
    \frac12 \Phi(u,d,P, z_0, r).
\eeq
Here we denote
\beqno
    &&\Phi(u,d,P,z_0, r)\\
&&:=\left(\int_{P_r(z_0)}|u|^4\right)^{\frac14}+\left(\int_{P_r(z_0)}|\nabla u|^2\right)^{\frac12}+\left(\int_{P_r(z_0)}|\nabla d|^4\right)^{\frac14}+\left(\int_{P_r(z_0)}|\Delta d|^2\right)^{\frac12}+\left(\int_{P_r(z_0)}|P|^2\right)^{\frac12}.
\eeqno
\end{lemma}

\begin{proof} First observe that since (\ref{ES2}) is invariant under translations and dilations, we have that
$$u_r(x,t)=r u(x_0+rx, t_0+r^2 t), \ d_r(x,t)=d(x_0+r x, t_0+r^2 t), \ P_r(x,t)=r^2P(x_0+r x, t_0+r^2 t),$$
 is a suitable weak solution of (\ref{ES2}) in $P_1(0)$. Thus  it suffices to prove the lemma
for $z_0=(0,0)$ and $r=1$.

We argue it by contradiction. Suppose that the lemma were false. Then there would exist $\epsilon_i\downarrow 0$ and a sequence of suitable weak solutions
$(u_i,d_i,P_i)\in (L^\infty_tL^2_x\cap L^2_tH^1_x)(P_1,\R^2)\times (L^\infty_t H^1_x\cap L^2_t H^2_x) (P_1,\mathbb S^2)\times L^2(P_1)$
of (\ref{ES2})  such that
\beq
    \Phi\left(u_i,d_i,P_i, (0,0), 1\right)=\epsilon_i\downarrow 0,
\eeq
but, for any $\theta\in(0,\frac12)$, it holds
\beq
    \Phi\left(u_i,d_i,P_i,(0,0), \theta\right)>\frac12 \epsilon_i.
\eeq
Now we define a blow-up sequence:
\beq
    \widetilde{u}_i=\frac{u_i}{\epsilon_i},\quad \widetilde{d}_i=\frac{d_i-(d_i)_{1}}{\epsilon_i},\quad \widetilde{P}_i=\frac{P_i}{\epsilon_i},
\eeq
where $\displaystyle (d_i)_1=\frac{1}{|P_1|}\int_{P_1} d_i$ is the average of $d_i$ over $P_1$.

It is easy to see that ($\widetilde{u}_i, \widetilde{d}_i, \widetilde{P}_i$) satisfies
\beq \label{lq2}
\begin{cases}
        \partial_t \widetilde{u}_i+\epsilon_i\widetilde{u}_i\cdot\nabla\widetilde{u}_i+\nabla \widetilde{P}_i=-\epsilon_i\nabla\cdot(\nabla \widetilde{d}_i\odot\nabla \widetilde{d}_i)+\epsilon_i^{-1}\nabla\cdot(\sigma^{L}(u_i, d_i)),\\
         \nabla\cdot \widetilde{u}_i=0,\\
       \partial_t \widetilde{d}_i+\epsilon_i{\widetilde{u}_i}\cdot\nabla{\widetilde{d}_i}-\epsilon_i^{-1}\Omega^i \hat{d_i}
+\frac{\lambda_2}{\lambda_1}\epsilon_i^{-1}A^i \hat{d_i}=\frac{1}{|\lambda_1|}\left(\Delta \widetilde{d}_i+\epsilon_i|\nabla \widetilde{d}_i|^2 d_i\right)
+\frac{\lambda_2}{\lambda_1}\epsilon_i^{-1}(\hat{d}_i^T A^i \hat{d}_i) d_i.
\end{cases}
\eeq
Moreover, we have
\beq \label{normalized}
\begin{cases}
   \Phi\left(\widetilde{u}_i,\widetilde{d}_i,\widetilde{P}_i, (0,0), 1\right)=1, \\
    \Phi\left(\widetilde{u}_i,\widetilde{d}_i,\widetilde{P}_i, (0,0), \theta\right)>\frac12.
\end{cases}
\eeq
It follows from the equation (\ref{lq2})$_3$ that
\beq \label{d-t-estimate}
\left\|\partial_t \widetilde{d}_i\right\|_{L^2(P_\frac34)}\le C\left(\|\widetilde{u}_i\cdot\nabla\widetilde{d}_i\|_{L^2(P_1)}
+\|\nabla \widetilde{u}_i\|_{L^2(P_1)}+\|\nabla\widetilde{d}_i\|_{L^4(P_1)}^2\right)\le C.
\eeq
After taking possible subsequences, we may assume that there exists
$$(\widetilde u, \widetilde d, \widetilde P)\in \left(L^\infty_tL^2_x\cap L^2_t H^1_x(P_1,\R^2)\right)
\times \left(L^\infty_t H^1_x\cap L^2_t H^2_x(P_1, \R^3)\right)\times L^2(P_1)$$
and a point $d_0\in \mathbb S^2$
such that
$$\begin{cases}\widetilde{u}_i\rightharpoonup \widetilde{u}\ \ {\rm in}\ \ L_t^4L_x^4(P_1)\cap L_t^2H_x^1(P_1),
\\
\widetilde{P}_i\rightharpoonup \widetilde{P} \ \ {\rm in}\ \ L_t^2L_x^2(P_1),\\
\widetilde{d}_i\rightharpoonup\widetilde{d}\ \ {\rm in}\ \ L_t^4W_x^{1,4}(P_1)\cap L_t^2H_x^2(P_1),\\
\widetilde{d}_i\rightharpoonup \widetilde{d} \ \ {\rm{in}}\ \ H^1(P_\frac34),\\
d_i\rightarrow d_0 \ {\rm{\ a.e.}}\ \ P_\frac34.
\end{cases}
$$
It is easy to check that
$$\epsilon_i^{-1}\Omega^i=\frac12\left(\nabla \widetilde{u}_i-(\nabla \widetilde{u}_i)^T\right)
\rightharpoonup \widetilde{\Omega}:=\frac12\left(\nabla \widetilde{u}-(\nabla \widetilde{u})^T\right) \
\ {\rm{in}} \ \ L^2(P_\frac34),$$
$$\epsilon_i^{-1} A^i=\frac12\left(\nabla \widetilde{u}_i+(\nabla \widetilde{u}_i)^T\right)
\rightharpoonup\widetilde{A}:=\frac12\left(\nabla \widetilde{u}+(\nabla \widetilde{u})^T\right) \
\ {\rm{in}} \ \ L^2(P_\frac34),$$
$$\epsilon_i^{-1}(\hat d_i^T A^i \hat d_i)\rightharpoonup \hat{d_0}^T \widetilde{A} \hat{d_0},
\ \ \epsilon_i^{-1}A^i\hat d_i\rightharpoonup \widetilde A \hat d_0\ \  {\rm{in}} \ \ L^2(P_\frac34),
$$
$$\epsilon_i^{-1} N^i=\epsilon_i^{-1}\left(\partial_t d_i+u_i\cdot\nabla d_i-\Omega^i \hat d_i\right)
\rightharpoonup \widetilde{\bf N}:=\partial_t \widetilde{d}-\widetilde{\Omega}\hat d_0
\ \ {\rm{in}}\ \ L^2(P_\frac34),
$$
while
\begin{eqnarray*}
&&\epsilon_i^{-1}\sigma^L(u_i, d_i):=\mu_1 (\hat{d}_i\otimes \hat {d}_i:(\epsilon_i^{-1}A^i)) \hat d_i\otimes \hat d_i
+\mu_2(\epsilon_i^{-1}N^i)\otimes \hat d_i +\mu_3 \hat d_i\otimes (\epsilon_i^{-1}N^i)+\mu_4 (\epsilon_i^{-1}A^i)\\
&&\qquad\qquad\qquad\ \ \ +\mu_5((\epsilon_i^{-1}A^i)\cdot\hat d_i)\otimes \hat d_i +\mu_6 \hat d_i\otimes ((\epsilon_i^{-1}A^i)\cdot\hat d_i)\\
&&\rightharpoonup{\widetilde{\sigma}}^L(\widetilde u, d_0)
:=\mu_1 (\hat d_0\otimes \hat d_0:\widetilde A) \hat d_0\otimes \hat d_0
+\mu_2\widetilde {\bf N}\otimes \hat d_0 +\mu_3 \hat d_0\otimes \widetilde {\bf N}+\mu_4 \widetilde A\\
&&\qquad\qquad\qquad\ \ \ +\mu_5(\widetilde A\cdot\hat d_0)\otimes \hat d_0 +\mu_6 \hat d_0\otimes (\widetilde A\cdot\hat d_0)
 \ \ \ \  {\rm{in}}\ \ \ L^2(P_\frac34).
\end{eqnarray*}
Since $|d_i|=1$, an elementary argument from the differential geometry implies that
\beq \label{tangent}
\widetilde d(x,t)\in T_{d_0}\mathbb S^2 \ \ {\rm{a.e.}}\ \ (x,t)\in P_\frac34.
\eeq
Therefore $(\widetilde{u},\widetilde{d},\widetilde{P})$ satisfies (\ref{tangent}) and the following linear system in $P_\frac34$:
\beq \label{lq3}
\begin{cases}
         \partial_t\widetilde{u}+\nabla \widetilde{P}=\nabla\cdot \left(\widetilde{\sigma}^L(\widetilde u, d_0)\right),\\
         \nabla\cdot \widetilde{u}=0,\\
        \partial_t\widetilde{d}-\widetilde\Omega \hat d_0+\frac{\lambda_2}{\lambda_1}\widetilde A\hat d_0
=\frac{1}{|\lambda_1|}\Delta{\widetilde{d}}+\frac{\lambda_2}{\lambda_1}\left(\hat{d_0}^T \widetilde A \hat{d_0}\right) d_0.
\end{cases}\eeq
By the lower semicontinuity, we have
\beq\label{renormal1}
    \Phi\left(\widetilde{u},\widetilde{d},\widetilde{P}, (0,0), 1\right)\leq 1.
\eeq

By the regularity lemma \ref{le:epsilon2} below, we know that ($\widetilde {u}, \widetilde {d}, \widetilde {P}$)
is smooth in $P_\frac12$ and there exists $0<\theta_0<\frac12$ such that
\beq\label{decay_limit}
\Phi\left(\widetilde{u},\widetilde{d},\widetilde{P}, (0,0), \theta_0\right)\leq C\theta_0<\frac14.
\eeq
In order to reach the desired contradiction, we need to apply the local energy inequality (\ref{local_energy_ineq1}) for ($\widetilde u_i, \widetilde d_i, \widetilde P_i$).
First, observe that the equation (\ref{lq2})$_1$  can be written as
\beq
\partial_t \widetilde{u}_i+\nabla \widetilde{P}_i-\frac{\mu_4}2\Delta \widetilde{u}_i
=g_i:=\Big[-\epsilon_i\nabla\cdot\left(\widetilde{u}_i\otimes\widetilde{u}_i+\nabla \widetilde{d}_i\odot\nabla\widetilde{d}_i\right)+\epsilon_i^{-1}\nabla\cdot\left(\sigma^L(u_i,d_i)-\mu_4 A^i\right)\Big].
\eeq
It follows from (\ref{normalized}) that $g_i\in L^2([-1,0], H^{-1}(B_1))$ and
\beq\nonumber
\Big\|g_i\Big\|_{L^2([-1,0], H^{-1}(B_1))}
&\lesssim& \Big[\|\widetilde u_i\|_{L^4(P_1)}^2+\|\nabla\widetilde{d}_i\|_{L^4(P_1)}^2
+\|\nabla\widetilde{u}_i\|_{L^2(P_1)}+\|\nabla^2 \widetilde{d}_i\|_{L^2(P_1)}\Big]\\
&\lesssim& \Phi\left(\widetilde{u}_i, \widetilde{d}_i, \widetilde{P}_i, (0,0), 1\right)\le C.
\nonumber
\eeq
Hence by the standard estimate on Stokes' system (cf. \cite{Temam}) we have that
$\displaystyle\partial_t \widetilde {u}_i\in L^2([-(\frac34)^2, 0], H^{-1}(B_\frac34))$ and
\beq\label{u-t-estimate}
\Big\|\partial_t\widetilde{u}_i\Big\|_{L^2\left([-(\frac34)^2, 0], H^{-1}(B_\frac34)\right)}
\lesssim\Big[\Phi\left(\widetilde{u}_i, \widetilde{d}_i, \widetilde{P}_i, (0,0), 1\right)
+\Big\|g_i\Big\|_{L^2([-1,0], H^{-1}(B_1))}\Big]\le C.
\eeq
It follows from (\ref{normalized}), (\ref{d-t-estimate}), and (\ref{u-t-estimate}) that we can
apply the Aubin-Lions compactness lemma (cf. \cite{Temam}) to conclude that, after taking possible
subsequences,
\beq \label{l2-strong-conv}
\widetilde{u}_i\rightarrow \widetilde{u},\  \nabla\widetilde{d}_i\rightarrow\nabla\widetilde{d}
\ \ {\rm in}\ \ L^2(P_{\frac34}).
\eeq
By Fubini's theorem, for any $\theta\in (0,\frac12)$ there exists $\tau_0\in (\theta^2,  4\theta^2)$
such that
$$
\int_{B_{2\theta}}\left(|\widetilde u_i-\widetilde u|^2+|\nabla \widetilde d_i-\nabla\widetilde d|^2\right)(x,-\tau_0)\,dx
\leq C\theta^{-2}\int_{P_{2\theta}}\left(|\widetilde u_i-\widetilde u|^2+|\nabla \widetilde d_i-\nabla\widetilde d|^2\right)
\leq C\theta^{-2} o(1).
$$
Here $o(1)$ denotes the constant such that $\displaystyle\lim_{i\rightarrow +\infty}o(1)=0$.
Since
$$\displaystyle \int_{B_{2\theta}}\left(|\widetilde u|^2+|\nabla\widetilde d|^2\right)(x,-\tau_0)\,dx
\leq C\theta^2,$$
we have
\beq\label{slice-d-estimate}
\int_{B_{2\theta}}\left(|\widetilde u_i|^2+|\nabla\widetilde d_i|^2\right)(x,-\tau_0)\,dx
\le C\left[\theta^2+\theta^{-2}o(1)\right].
\eeq
Since $(u_i, d_i, P_i)$ satisfies the local energy inequality (\ref{local_energy_ineq1}), we see that by rescalings $(\widetilde u_i, \widetilde d_i, \widetilde P_i)$
satisfies the following local energy inequality: for any $-\tau_0\le t\le 0$ and any $\eta\in C_0^\infty(B_1)$,
\beq\label{local_energy_ineq2}
&&\int_{\R^2}\eta^2(|\widetilde u_i|^2+|\nabla \widetilde d_i|^2)(t)+\int_{-\tau_0}^{t}\int_{\R^2}\eta^2\Big[\mu_4|\nabla \widetilde u_i|^2
+\frac2{|\lambda_1|}
|\Delta \widetilde d_i+\epsilon_i|\nabla \widetilde d|^2 d_i|^2\Big]\nonumber\\
&\le& \int_{\R^2}\eta^2(|\widetilde u_i|^2+|\nabla \widetilde d_i|^2)(-\tau_0)\nonumber\\
&&+C\int_{-\tau_0}^{t}\int_{\R^2}\Big[(\epsilon_i|\widetilde u_i|^2+|\nabla \widetilde u_i|+\epsilon_i|\widetilde u_i||\nabla \widetilde d_i|
+\epsilon_i|\nabla \widetilde d_i|^2+|\nabla^2 \widetilde d_i|+|\widetilde P_i|)|\widetilde u_i||\nabla(\eta^2)|\nonumber\\
&&\qquad\qquad\ \ \ \ \ +(|\nabla \widetilde u_i|+\epsilon_i|\widetilde u_i||\nabla\widetilde d_i|+\epsilon_i|\nabla \widetilde d_i|^2+
|\nabla^2\widetilde d_i|)|\nabla \widetilde d_i||\nabla(\eta^2)|\Big].
\eeq
By the weak and strong convergence properties for $(\widetilde u_i, \widetilde d_i, \widetilde P_i)$ listed as above,
 we have that, as $i\rightarrow +\infty$,
\beq
&&{\bf E}_i:=\int_{-\tau_0}^{t}\int_{\R^2}\Big[(\epsilon_i|\widetilde u_i|^2+|\nabla \widetilde u_i|+\epsilon_i|\widetilde u_i||\nabla \widetilde d_i|
+\epsilon_i|\nabla \widetilde d_i|^2+|\nabla^2 \widetilde d_i|+|\widetilde P_i|)|\widetilde u_i||\nabla(\eta^2)|\nonumber\\
&&\qquad\qquad+(|\nabla \widetilde u_i|+\epsilon_i|\widetilde u_i||\nabla\widetilde d_i|+\epsilon_i|\nabla \widetilde d_i|^2+
|\nabla^2\widetilde d_i|)|\nabla \widetilde d_i||\nabla(\eta^2)|\Big]\nonumber\\
&&\rightarrow
{\bf E}:=\int_{-\tau_0}^{t}\int_{\R^2}\Big[(|\nabla \widetilde u|+|\nabla^2 \widetilde d|+|\widetilde P|)|\widetilde u||\nabla(\eta^2)|
+(|\nabla \widetilde u|+|\nabla^2\widetilde d|)|\nabla \widetilde d||\nabla(\eta^2)|\Big].
\eeq
Now we choose $\eta\in C^\infty_0(B_1)$ such that
$$0\le\eta\le 1, \ \eta\equiv 1 \ {\rm{in}}\ B_{\sqrt{\tau_0}},\ \eta\equiv 0\ {\rm{outside}}\
B_{2\sqrt{\tau_0}},\ {\rm{and}}\ |\nabla\eta|\le C\tau_0^{-\frac12}\le C\theta^{-1}.$$
Then we have that for $\theta\in (0,\frac14)$,
\begin{eqnarray*}
|{\bf E}|&\lesssim& \theta^{-1}\int_{P_{2\theta}}\Big[(|\nabla \widetilde u|+|\nabla^2 \widetilde d|+|\widetilde P|)|\widetilde u|
+(|\nabla \widetilde u|+|\nabla^2\widetilde d|)|\nabla \widetilde d|\Big]\\
&\lesssim& \Big[\left(\|\nabla^2 \widetilde d\|_{L^2(P_{2\theta})}+\|\nabla\widetilde u\|_{L^2(P_{2\theta})}+\|\widetilde P\|_{L^2(P_{2\theta})}\right)
\left(\|\nabla\widetilde d\|_{L^4(P_{2\theta})}+\|\widetilde u\|_{L^4(P_{2\theta})}\right)\Big]\\
&\lesssim& \Big[\Phi\big(\widetilde u,\widetilde d, \widetilde P, (0,0), 2\theta\big)\Big]^2
\le C\theta^2
\end{eqnarray*}
so that
\beq\label{error_estimate}
|{\bf E}_i|\le C(\theta^2+o(1)).
\eeq
Substituting (\ref{slice-d-estimate}) and (\ref{error_estimate}) into (\ref{local_energy_ineq2}) yields
\beq\label{local_energy_ineq3}
\sup_{-\theta^2\le t\le 0} \int_{B_\theta}(|\widetilde u_i|^2+|\nabla \widetilde d_i|^2)(t)
+\int_{P_\theta}\Big[\mu_4|\nabla \widetilde u_i|^2
+\frac2{|\lambda_1|}
|\Delta \widetilde d_i+\epsilon_i|\nabla \widetilde d|^2 d_i|^2\Big]\le C\left(\theta^2+\theta^{-2} o(1)\right).
\eeq
Since
$$\int_{P_\theta} |\Delta \widetilde d_i|^2\le \int_{P_\theta}
|\Delta \widetilde d_i+\epsilon_i|\nabla \widetilde d|^2 d_i|^2+\epsilon_i^2\int_{P_{\theta}}|\nabla\widetilde d_i|^4
\le  \int_{P_\theta}
|\Delta \widetilde d_i+\epsilon_i|\nabla \widetilde d|^2 d_i|^2+C\epsilon_i^2,
$$
and by the $H^{2}$-estimate
$$\int_{P_{\frac{\theta}2}} |\nabla^2 \widetilde d_i|^2\lesssim \int_{P_{\theta}} |\Delta \widetilde d_i|^2+\theta^{-2}\int_{P_{\theta}}|\nabla\widetilde d_i|^2
\lesssim\int_{P_{\theta}} |\Delta \widetilde d_i|^2+\theta^2+\theta^{-2} o(1) ,$$
we obtain
\beq\label{local_energy_ineq4}
\sup_{-\frac14\theta^2\le t\le 0} \int_{B_{\frac\theta{2}}}(|\widetilde u_i|^2+|\nabla \widetilde d_i|^2)(t)
+\int_{P_{\frac\theta{2}}}\left(|\nabla \widetilde u_i|^2+|\nabla^2\widetilde d_i|^2\right)
\le C\left(\theta^2+\epsilon_i^2+\theta^{-2} o(1)\right).
\eeq

Recall Ladyzhenskaya's inequality in $\R^2$ (cf. \cite{LSU}):
\beq\label{lady}
\int_{B_r}|f|^4\lesssim \int_{B_{2r}}|f|^2\int_{B_{2r}}\left(r^{-2}|f|^2+|\nabla f|^2\right),
\ \forall f\in H^1(B_{2r}).
\eeq
Applying (\ref{lady}) to $\widetilde u_i$ and $\nabla\widetilde d_i$ and integrating over $t$-variable and using (\ref{local_energy_ineq4}),  we obtain
\beq \label{u-d-4estimate}
\int_{P_{\frac{\theta}4}}\left(|\widetilde u_i|^4+|\nabla\widetilde d_i|^4\right)
&\lesssim& \left(\sup_{t\in [-\frac14\theta^2, 0]}\int_{B_{\frac\theta{2}}}|\widetilde u_i|^2
+|\nabla\widetilde d_i|^2\right)\int_{P_{\frac{\theta}2}}\left[\theta^{-2}(|\widetilde u_i|^2
+|\nabla\widetilde d_i|^2)+|\nabla \widetilde u_i|^2
+|\nabla^2 \widetilde d_i|^2\right]\nonumber\\
&\le& C\left(\theta^4+\epsilon_i^4+\theta^{-4}o(1)\right).
\eeq
To estimate $\widetilde P_i$, let's take divergence of the equation (\ref{lq2})$_1$ to get
\beq\label{P-equation}
\Delta\widetilde P_i=-\epsilon_i(\nabla\cdot)^2\left(\widetilde u_i\otimes\widetilde u_i+\nabla\widetilde d_i\odot\nabla\widetilde d_i\right)
+\epsilon_i^{-1}(\nabla\cdot)^2\left(\sigma^L(u_i,d_i)\right)\ \ \ {\rm{in}}\ \ \ B_\theta.
\eeq
Let $\phi\in C_0^\infty(\R^2)$ such that
$$0\le\phi\le 1,\  \phi\equiv 1 \ {\rm{in}}\ B_{\frac{\theta}8},\  \phi\equiv 0\ {\rm{outside}} \ B_{\frac{\theta}4}, \ {\rm{and}} \ |\nabla\phi|\lesssim \theta^{-1}.$$
Define $\widetilde Q_i$ by
$$\widetilde Q_i(x,t)
=-\int_{\R^2}\nabla^2_y G(x-y): \phi^2(y)\left[\epsilon_i (\widetilde u_i\otimes\widetilde u_i+\nabla\widetilde d_i\odot\nabla\widetilde d_i)-\epsilon_i^{-1}
\sigma^L(u_i, d_i)\right](y,t)\,dy,$$
where $G$ is the fundamental solution of the Laplace equation on $\R^2$. Then we have
$$
\Delta \widetilde Q_i= -(\nabla\cdot)^2\Big[\phi^2\left(\epsilon_i(\widetilde u_i\otimes\widetilde u_i+\nabla\widetilde d_i\odot\nabla\widetilde d_i)
-\epsilon_i^{-1}\sigma^L(u_i,d_i)\right)\Big] \ \ {\rm{in}}\ \ \R^2.
$$
By Calderon-Zygmund's $L^2$-theory we have
\beq\label{Qi-estimate1}
\|\widetilde Q_i\|_{L^2(\R^2)}^2&\lesssim& \Big[\epsilon_i(\|\widetilde u_i\|_{L^4(B_{\frac{\theta}4})}^4+\|\nabla\widetilde d_i\|_{L^4(B_{\frac{\theta}4})}^4)
+\|\epsilon_i^{-1}\sigma^L(u_i,d_i)\|_{L^2(B_{\frac{\theta}4})}^2\Big]\nonumber\\
&\lesssim&
\Big[\epsilon_i(\|\widetilde u_i\|_{L^4(B_{\frac{\theta}4})}^4+\|\nabla\widetilde d_i\|_{L^4(B_{\frac{\theta}4})}^4)
+\|\nabla\widetilde u_i\|_{L^2(B_{\frac{\theta}4})}^2+\|\epsilon_i^{-1}N^i\|_{L^2(B_{\frac{\theta}4})}^2\Big]\nonumber\\
&\lesssim&
\Big[\|\widetilde u_i\|_{L^4(B_{\frac{\theta}4})}^4+\|\nabla\widetilde d_i\|_{L^4(B_{\frac{\theta}4})}^4
+\|\nabla\widetilde u_i\|_{L^2(B_{\frac{\theta}4})}^2+\|\nabla^2\widetilde d_i\|_{L^2(B_{\frac{\theta}4})}^2\Big].
\eeq
Integrating (\ref{Qi-estimate1}) over $t\in [-\theta^2, 0]$, and using (\ref{local_energy_ineq4}) and (\ref{u-d-4estimate}), we obtain
\beq\label{Qi-estimate2}
\int_{-\theta^2}^0\int_{\R^2}|\widetilde Q_i|^2\lesssim \int_{P_{\frac{\theta}4}}\left(|\widetilde u_i|^4+|\nabla \widetilde d_i|^4+|\nabla\widetilde u_i|^2
+|\nabla^2 \widetilde d_i|^2\right)\le C\left(\theta^4+\epsilon_i^4+\theta^{-4}o(1)\right).
\eeq
Set $\widetilde R_i=\widetilde P_i-\widetilde Q_i$ in $P_\theta$. Then we have
$$\Delta \widetilde R_i(t)=0 \ {\rm{in}}\ B_{\frac{\theta}4}, \ \forall t\in [-\theta^2, 0],$$
so that by the standard estimate of harmonic functions and (\ref{Qi-estimate2}) we have
\beq\label{Ri-estimate}
\int_{P_{\theta^2}}|\widetilde R_i|^2\lesssim\theta^2\int_{P_{\frac{\theta}4}}|\widetilde R_i|^2
\le C\theta^2 \int_{P_{\frac{\theta}4}}(|\widetilde P_i|^2+|\widetilde Q_i|^2)\le C\Big[\theta^2+\left(\theta^4+\epsilon_i^4+\theta^{-4}o(1)\right)\Big].
\eeq
Putting (\ref{Qi-estimate2}) together with (\ref{Ri-estimate}) yields
\beq\label{Pi-estimate}
\int_{P_{\theta^2}}|\widetilde P_i|^2\le C\Big[\theta^2+\left(\theta^4+\epsilon_i^4+\theta^{-4}o(1)\right)\Big].
\eeq
Combining all these estimates (\ref{local_energy_ineq4}), (\ref{u-d-4estimate}), and (\ref{Pi-estimate}), we obtain
\beq
\Phi\left(\widetilde u_i, \widetilde d_i, \widetilde P_i, (0,0), \theta^2\right)
\le C\Big[\theta+\epsilon_i+\theta^{-1}o(1)\Big] \le \frac14,
\eeq
provided that we first choose sufficiently small $\theta$ and  then choose sufficiently large $i$.
This gives the desired contradiction. The proof is complete.
\end{proof}
\medskip

The following lemma plays an important role in the blow-up process, which may have its own interest.

\begin{lemma} \label{le:epsilon2}  Assume  (\ref{necessary}), (\ref{parodi}), and (\ref{Leslie_condition}) hold.
For any point $d_0\in\mathbb S^2$,
if $\widetilde{u}\in \left(L^\infty_t L^2_x\cap L^2_tH^1_x\right)(P_\frac34,\R^2)$,
$\widetilde{d}\in \left(L^\infty_tH^{1}_x\cap L^2_tH^2_x\right)(P_\frac34,T_{d_0}\mathbb S^2)$,
and $\widetilde{P}\in L^2(P_\frac34)$ solves the linear system (\ref{lq3}) and satisfies the condition (\ref{renormal1}),
then $(\widetilde{u},\widetilde{d},\widetilde{P})\in C^\infty(P_{\frac12})$
and satisfies the following estimate:
\beq\label{decay1}
  \Phi\left(\widetilde{u},\widetilde{d},\widetilde{P}, (0,0), \theta\right)\leq C\theta,\ \forall \ \theta\in(0,\frac12).
\eeq
\end{lemma}
\begin{proof} To simplify the notations, we write $(u,d, P)$ for ($\widetilde u, \widetilde d, \widetilde P$) in the proof below.
The argument is based on the higher order local energy inequality argument.

Taking $\frac{\partial}{\partial x_i}$ of the linear system (\ref{lq3}) yields
\beq
\label{lq4}
\begin{cases}
         \partial_t u_{x_i}+\nabla {P_{x_i}}=\nabla\cdot(\widetilde{\sigma}^L(u, d_0))_{x_i},\\
         \nabla\cdot u_{x_i}=0,\\
        \partial_t d_{x_i}-\Omega_{x_i}\hat d_0+\frac{\lambda_2}{\lambda_1}A_{x_i}\hat d_0=
\frac{1}{|\lambda_1|}\Delta d_{x_i}+\frac{\lambda_2}{\lambda_1}\left(\hat d_0^TA_{x_i} \hat d_0\right) d_0.
\end{cases}
\eeq
For any $\eta\in C_0^\infty(B_1)$, multiplying the equation (\ref{lq4})$_1$ by $u_{x_i}\eta^2$ and the equation (\ref{lq4})$_3$
by  $\Delta d_{x_i}\eta^2$ and integrating the resulting equations over $B_1$, we obtain\footnote{Strictly speaking, we first need to take finite quotient $D_h^i$
of the system (\ref{lq3}) and then multiply the first equation and the third equation of the resulting equations by $D_h^i u(x,t)\eta^2=\frac{u(x+he_i, t)-u(x,t)}{h}$ and$D_h^i(\Delta d)\eta^2$  respectively for $h>0$ and $i=1,2$, with $e_1=(1,0)$ and $ e_2=(0,1)$. Then the desired estimate
follows from the estimates on the finite quotients by sending $h$ to zero.}
\beq
    \label{u}
   && \frac{d}{dt}\int_{B_1}|\nabla u|^2\eta^2=\\
&&-2\Big[\int_{B_1}P\left[\Delta u\cdot\nabla(\eta^2)+\nabla u:\nabla^2(\eta^2)\right]
+\int_{B_1}\left(\widetilde{\sigma}^L(u,d_0)\right)_{x_i} : u_{x_i}\otimes\nabla(\eta^2)+\int_{B_1}\eta^2\left(\widetilde{\sigma}^L(u,d_0)\right)_{x_i} : \nabla u_{x_i}\Big],\nonumber
\eeq
\beq
    \label{d}
    &&\frac{d}{dt}\int_{B_1}|\nabla^2 d|^2\eta^2+\frac{2}{|\lambda_1|}\int_{B_1}|\Delta\nabla d|^2\eta^2\nonumber\\
&=&-2\int_{B_1}\partial_t d_{x_i}\cdot\nabla d_{x_i}\cdot\nabla(\eta^2) -
2\int_{B_1}\left(\langle\Omega_{x_i} \hat d_0-\frac{\lambda_2}{\lambda_1}A_{x_i}\hat d_0,\Delta \hat d_{x_i}\rangle
+\langle\frac{\lambda_2}{\lambda_1}(\hat d_0^TA_{x_i}\hat d_0)d_0,
\Delta d_{x_i}\rangle\right)\eta^2.\nonumber\\
&=&-2\int_{B_1}\partial_t d_{x_i}\cdot\nabla d_{x_i}\cdot\nabla(\eta^2) -
2\int_{B_1}\left(\langle\Omega_{x_i} \hat d_0-\frac{\lambda_2}{\lambda_1}A_{x_i}\hat d_0,\Delta \hat d_{x_i}\rangle\right)\eta^2,
\eeq
where we have used in the last step the fact $d\in T_{d_0}\mathbb S^2$ in order to deduce that $\langle d_0, \Delta d_{x_i}\rangle =0$  a.e. in $B_1$.

Similar to the calculations in the proof of lemma \ref{global_energy_ineq}, we have
\begin{eqnarray}
&&\int_{B_1}\eta^2\left(\widetilde{\sigma}^L(u,d_0)\right)_{x_i} : \nabla u_{x_i}\nonumber\\
&=&\int_{B_1}\eta^2\Big[\mu_1 (\hat d_0\otimes \hat d_0:A_{x_i}) \hat d_0\otimes \hat d_0
+\mu_2\widetilde {\bf N}_{x_i}\otimes \hat d_0 +\mu_3 \hat d_0\otimes \widetilde {\bf N}_{x_i}+\mu_4 A_{x_i}\nonumber\\
&&\qquad\ \ +\mu_5(A_{x_i}\cdot\hat d_0)\otimes \hat d_0 +\mu_6 \hat d_0\otimes (A_{x_i}\cdot\hat d_0)\Big]:\left(A_{x_i}+\Omega_{x_i}\right)\nonumber\\
&=& \int_{B_1}\eta^2 \Big[\mu_1 (\hat d_0^T A_{x_i}\hat d_0)^2+\mu_4 |A_{x_i}|^2-\lambda_2\widetilde{\bf N}_{x_i}\cdot(A_{x_i}\cdot \hat d_0)
+\lambda_1\widetilde{\bf N}_{x_i}\cdot(\Omega_{x_i}\cdot\hat d_0) \nonumber\\
&&\qquad\ \ +(\mu_5+\mu_6)|A_{x_i}\cdot \hat d_0|^2+\lambda_2(A_{x_i}\cdot \hat d_0)(\Omega_{x_i}\cdot \hat d_0)\Big]. \label{sigma_identity}
\end{eqnarray}
By the equation (\ref{lq4})$_3$, we have
$$
\lambda_1\widetilde{\bf N}_{x_i}\cdot(\Omega_{x_i}\cdot\hat d_0)
=-\lambda_2(A_{x_i}\cdot\hat{d}_0)(\Omega_{x_i}\cdot\hat{d}_0)-\langle\Omega_{x_i}\hat{d}_0, \Delta \hat {d}_{x_i}\rangle,
$$
and
$$
-\lambda_2\widetilde{\bf N}_{x_i}\cdot(A_{x_i}\cdot\hat d_0)
=\frac{\lambda_2^2}{\lambda_1}|A_{x_i}\cdot\hat d_0|^2-\frac{\lambda_2^2}{\lambda_1}(\hat d_0^T A_{x_i} \hat d_0)^2
+\frac{\lambda_2}{\lambda_1}\langle A_{x_i}\hat d_0, \Delta \hat d_{x_i}\rangle.
$$
Substituting these identities into (\ref{sigma_identity}), we obtain
\begin{eqnarray}
&&\int_{B_1}\eta^2\left(\widetilde{\sigma}^L(u,d_0)\right)_{x_i} : \nabla u_{x_i}\nonumber\\
&=& \int_{B_1}\eta^2 \Big[(\mu_1-\frac{\lambda_2^2}{\lambda_1}) (\hat d_0^T A_{x_i}\hat d_0)^2+\mu_4 |A_{x_i}|^2
+\langle\frac{\lambda_2}{\lambda_1}A_{x_i}\hat{d}_0-\Omega_{x_i}\hat d_0, \Delta \hat d_{x_i}\rangle \nonumber\\
&&\qquad\ \ +(\mu_5+\mu_6+\frac{\lambda_2^2}{\lambda_1})|A_{x_i}\cdot \hat d_0|^2\Big]. \label{sigma_identity1}
\end{eqnarray}
Putting (\ref{sigma_identity1}) into (\ref{u}) and adding the resulting (\ref{u}) with (\ref{d}),  we have, by (\ref{Leslie_condition}),
\beq
    \nonumber
    \label{u,d}
    &&\frac{d}{dt}\int_{B_1}\left(|\nabla{u}|^2+|\nabla^2{d}|^2\right)\eta^2 +\int_{B_1}\left(\mu_4|\nabla^2 u|^2+\frac{2}{|\lambda_1|}|\nabla^3 d|^2\right)\eta^2
    \\&=&-2\int_{B_1}\eta^2 \Big[(\mu_1-\frac{\lambda_2^2}{\lambda_1}) (\hat d_0^T A_{x_i}\hat d_0)^2
+(\mu_5+\mu_6+\frac{\lambda_2^2}{\lambda_1})|A_{x_i}\cdot \hat d_0|^2\Big]
 \nonumber\\
&&-2\int_{B_1}P\left[\Delta u\cdot\nabla(\eta^2)+\nabla u:\nabla^2(\eta^2)\right]
-2\int_{B_1}\left(\widetilde{\sigma}^L(u,d_0)\right)_{x_i} : u_{x_i}\otimes\nabla(\eta^2)\nonumber\\
&&+\int_{B_1}\Big(\mu_4\nabla^2 u\cdot\nabla(\eta^2):\nabla u-2\partial_t d_{x_i}\cdot\nabla d_{x_i}\cdot\nabla(\eta^2)
-\frac{2}{|\lambda_1|}\nabla(\nabla d)\otimes\nabla(\nabla d):\nabla^2(\eta^2)\Big)\nonumber\\
&\le& -2\int_{B_1}P\left[\Delta u\cdot\nabla(\eta^2)+\nabla u:\nabla^2(\eta^2)\right]
-2\int_{B_1}\left(\widetilde{\sigma}^L(u,d_0)\right)_{x_i} : u_{x_i}\otimes\nabla(\eta^2)+\nonumber\\
&&\int_{B_1}\Big[\mu_4\nabla^2 u\cdot\nabla(\eta^2):\nabla u-2\partial_t d_{x_i}\cdot\nabla d_{x_i}\cdot\nabla(\eta^2)
+\frac{2}{|\lambda_1|}(\nabla^3 d\cdot(\nabla d)_{x_i}+\Delta\nabla d\cdot(\nabla d)_{x_i})(\eta^2)_{x_i}\Big]\nonumber\\
&:=& R_1+R_2+R_3.
\eeq
Now we estimate each term of the right hand side of (\ref{u,d}) as follows.
\begin{eqnarray*}
    |R_1|&\lesssim& \int_{B_1}|P|(|\nabla^2 u|^2\eta|\nabla\eta|+|\nabla u|(|\nabla^2\eta|+|\nabla\eta|^2))\\
&\le& \frac{\mu_4}6\int_{B_1}|\nabla^2 u|^2\eta^2+C\int_{B_1}(|P|^2+|\nabla u|^2)(|\nabla\eta|^2+|\nabla^2\eta|),
\end{eqnarray*}
\begin{eqnarray*}
 |R_2|&\lesssim&\int_{B_1}(|\nabla^2 u|+|\nabla^3 d|)|\nabla u|\eta|\nabla\eta|\\
&\leq& \frac{\mu_4}{6}\int_{B_1}|\nabla^2 u|^2\eta^2+\frac{1}{2|\lambda_1|}\int_{B_1}|\nabla^3 d|^2\eta^2
+C\int_{B_1}|\nabla u|^2|\nabla\eta|^2,
\end{eqnarray*}
and
\begin{eqnarray*}
|R_3|&\lesssim &\int_{B_1}|\nabla^2 u||\nabla u|\eta|\nabla\eta|+(|\nabla^2 u|+|\nabla^3 d|)|\nabla^2 d|\eta|\nabla\eta|\\
&\le& \frac{\mu_4}6\int_{B_1}|\nabla^2 u|^2\eta^2+\frac{1}{2|\lambda_1|}\int_{B_1}|\nabla^3 d|^2\eta^2
+C\int_{B_1}(|\nabla u|^2+|\nabla^2 d|^2)|\nabla\eta|^2.
\end{eqnarray*}
Putting these estimates into (\ref{u,d}), we obtain
\beq\label{u-d-1}
&&\frac{d}{dt}\int_{B_1}\left(|\nabla{u}|^2+|\nabla^2{d}|^2\right)\eta^2 +\int_{B_1}\left(\frac{\mu_4}2|\nabla^2 u|^2+\frac{1}{|\lambda_1|}|\nabla^3 d|^2\right)\eta^2 \nonumber\\
&&\le C\int_{B_1}\Big[(|P|^2+|\nabla u|^2)(|\nabla\eta|^2+|\nabla^2\eta|)+|\nabla^2 d|^2|\nabla\eta|^2\Big],
\eeq
By Fubini's theorem, there exists $t_*\in [-\frac14,0]$ such that
\beq\nonumber
    \int_{B_1}\left(|\nabla{u}|^2+|\nabla^2{d}|^2\right)\eta^2(t_*)\leq 8\int_{P_1}\left(|\nabla u|^2+|\nabla^2 d|^2\right)\eta^2.
\eeq
Integrating (\ref{u-d-1}) over $t\in[t_*,0]$ yields that
\beq
    \nonumber
    &&\sup_{-\frac14\le t\le 0}\int_{B_1}\left(|\nabla u|^2+|\nabla^2 d|^2\right)\eta^2(t)
+\int_{-\frac14}^0\int_{B_1}\left(|\nabla^2{u}|^2+|\nabla^3{d}|^2\right)\eta^2
    \\&&\leq C\int_{P_1}\Big[(|P|^2+|\nabla u|^2)(|\nabla\eta|^2+|\nabla^2\eta|)+|\nabla^2 d|^2|\nabla\eta|^2\Big]
+ C\int_{P_1}\left(|\nabla u|^2+|\nabla^2 d|^2\right)\eta^2.\label{u-d-2}
\eeq
For the pressure $P$, taking divergence of the equation (\ref{lq4})$_1$ yields that for any $-1\le t\le 0$,
\beq\label{pressure-equation}
\Delta P_{x_i}=(\nabla\cdot)^2\left(\widetilde\sigma^L(u, d_0)\right)_{x_i} \ \ {\rm{in}}\ \ B_\frac34.
\eeq
Similar to the pressure estimates obtained in the proof of lemma \ref{le:epsilon}, we have
\beq \label{pressure-estimate1}
\int_{P_\frac14}|\nabla P|^2\lesssim \int_{P_\frac5{16}}|\widetilde\sigma^L(u,d_0)_{x_i}|^2+|P|^2
\lesssim \int_{P_\frac38}\left(|\nabla^2 u|^2+|\nabla^3 d|^2+|P|^2\right).
\eeq
Let $\eta\in C_0^1(B_1)$ be a cut-off function of $B_\frac38$, i.e. $\eta\equiv 1$ in $B_\frac38$, $\eta\equiv 0$ outside $B_\frac34$,
$0\le \eta\le 1$, and $|\nabla\eta|+|\nabla^2\eta|\le 16$. Then, by combining (\ref{u-d-2}) with (\ref{pressure-estimate1}), we obtain
\beq
    \nonumber
    &&\sup_{-(\frac14)^2\le t\le 0}\int_{B_\frac14}\left(|\nabla u|^2+|\nabla^2 d|^2\right)(t)
+\int_{P_\frac14}\left(|\nabla^2{u}|^2+|\nabla^3{d}|^2+|\nabla P|^2\right)
    \\&&\leq C\int_{P_\frac34}\Big[|P|^2+|\nabla u|^2+|\nabla^2 d|^2\Big].
\label{u-d-3}
\eeq
It turns out that the above energy method can be extended to any high order. Here we only give a sketch of the proof.
In fact, if we denote
$\nabla^\alpha=\frac{\partial^k}{\partial x^\alpha}$ as the $k$-th order derivative for any multiple index
$\alpha=(\alpha_1,\alpha_2)$ ($k=|\alpha|=\alpha_1+\alpha_2\ge 2$), and take $\nabla^\alpha$ of the system (\ref{lq3}),
then we obtain
\beq
\label{lq5}
\begin{cases}
         \partial_t (\nabla^\alpha u)+\nabla ({\nabla^\alpha P})=\nabla\cdot(\nabla^\alpha(\widetilde{\sigma}^L(u, d_0))),\\
         \nabla\cdot (\nabla^\alpha u)=0,\\
        \partial_t (\nabla^\alpha d)-(\nabla^\alpha \Omega) \hat d_0+\frac{\lambda_2}{\lambda_1}(\nabla^\alpha A)\hat d_0=
\frac{1}{|\lambda_1|}\Delta (\nabla^\alpha d)+\frac{\lambda_2}{\lambda_1}\left(\hat d_0^T(\nabla^\alpha A) \hat d_0\right) d_0.
\end{cases}
\eeq
Multiplying (\ref{lq5})$_1$ by $(\nabla^\alpha u) \eta^2$ and (\ref{lq5})$_3$ by $\Delta (\nabla^\alpha d)\eta^2$ and integrating the resulting
equations over $B_1$, and repeating the above calculations and cancelations, we would obtain
\beq\label{kth-estimate}
&&\frac{d}{dt}\int_{B_1}\left(|\nabla^k u|^2+|\nabla^{k+1}d|^2\right)\eta^2+
\int_{B_1}\left(\frac{\mu_4}2|\nabla^{k+1}u|^2+\frac{1}{|\lambda_1|}|\nabla^{k+2}d|^2\right)\eta^2\nonumber\\
&&\le C\int_{B_1}\Big[(|\nabla^{k-1}P|^2+|\nabla^k u|^2)(|\nabla\eta|^2+|\nabla^2 \eta|)+|\nabla^{k+1}d|^2|\nabla\eta|^2\Big].
\eeq
For $P$, since
\beq\label{pressure-equation1}
\Delta (\nabla^\alpha P)=(\nabla\cdot)^2\left(\nabla^\alpha(\widetilde\sigma^L(u, d_0))\right) \ \ {\rm{in}}\ \ B_\frac34,
\eeq
we have
\beq\label{pressure-estimate2}
\int_{P_\frac14} |\nabla^k P|^2\lesssim \int_{P_\frac38}\left(|\nabla^{k-1}P|^2+|\nabla^{k+1}u|^2+|\nabla^{k+2}d|^2\right).
\eeq
Following the same lines of proof as above, we can choose suitable time slice $t_*\in (-\frac14,0)$ such that
\beq\nonumber\label{slice1}
    \int_{B_1}\left(|\nabla^k{u}|^2+|\nabla^{k+1}{d}|^2\right)\eta^2(t_*)\leq 8\int_{P_1}\left(|\nabla^k u|^2+|\nabla^{k+1} d|^2\right)\eta^2.
\eeq
By choosing suitable test functions similar to the above ones, we can reach that for any $k\ge 2$, it holds
\beq
    \nonumber
    &&\sup_{-(\frac14)^2\le t\le 0}\int_{B_\frac14}\left(|\nabla^k u|^2+|\nabla^{k+1} d|^2\right)(t)
+\int_{P_\frac14}\left(|\nabla^{k+1}{u}|^2+|\nabla^{k+2}{d}|^2+|\nabla^k P|^2\right)
    \\&&\leq C\int_{P_\frac34}\Big[|\nabla^{k-1}P|^2+|\nabla^k u|^2+|\nabla^{k+1}d|^2\Big].
\label{u-d-4}
\eeq
It is clear that with suitable adjusting of the radius, we see that (\ref{u-d-4}) and (\ref{u-d-3}) implies that
\beq
    \nonumber
    &&\sup_{-(\frac14)^2\le t\le 0}\int_{B_\frac14}\left(|\nabla^k u|^2+|\nabla^{k+1} d|^2\right)(t)
+\int_{P_\frac14}\left(|\nabla^{k+1}{u}|^2+|\nabla^{k+2}{d}|^2+|\nabla^k P|^2\right)
    \\&&\leq C\int_{P_\frac34}\Big[|P|^2+|\nabla u|^2+|\nabla^2d|^2\Big]
\label{u-d-5}
\eeq
holds for all $k\ge 1$.

Now we can apply the regularity theory for both the linear Stokes equations (c.f. \cite{Temam}) and the linear heat equation
(cf. \cite{LSU}) to conclude that $(u,d)\in C^\infty(P_{\frac12})$.
Furthermore, apply the elliptic estimate for the pressure equation (\ref{pressure-equation}), we see that
$P\in C^\infty(P_{\frac12})$ (first we have $\nabla^k{P}\in C^0(P_{\frac12})$, then note that $\partial^l_t{P}$ also satisfies a similar elliptic equation, so that $\nabla^k\partial_t^l P\in C^0(P_{\frac12})$). Therefore $(u,d,P)\in C^\infty(P_{\frac12})$ and the desired estimate (\ref{decay1}) holds.
The proof of lemma \ref{le:epsilon2} is complete.
\end{proof}

In order to show the smoothness of solutions to (\ref{ES2}) under the condition (\ref{small_condition}),
we need to iterate the decay inequality (\ref{decay_ineq}) and establish
higher integrability of $(u,\nabla d)$ by applying the techniques of Morrey space estimates for Riesz potentials, similar to that by Hineman-Wang \cite{Hine-Wang}.

\begin{lemma}
  \label{high_integrability} Assume that the conditions (\ref{necessary}),
(\ref{parodi}), and (\ref{Leslie_condition}) hold. For any $0<T\le +\infty$ and a bounded domain $O\subset\R^2$, there exists $\epsilon_0>0$ such that
  $(u,P, d)\in L^\infty_t L^2_x\cap L^2_tH^1_x(O\times [0,T])\times L^2(O\times [0,T])\times L^2_tH^2_x(O\times [0,T],\mathbb S^2)$
is a suitable weak solution of (\ref{ES2}), and satisfies, for $z_0=(x_0,t_0)\in O\times (0,T)$ and $P_{r_0}(z_0)\subset O\times (0, T)$,
  \begin{equation}\label{eps00}
   \Phi\left(u, d, P, z_0, r_0\right)
\leq \epsilon_0,
  \end{equation}
 then $(u,\nabla d) \in L^q_{\rm{loc}}(P_{r_0}(z_0))$ for any $1<q<+\infty$. Moreover, it holds
  \begin{equation}
  \label{lq:estimate}
  \|(u,\nabla d)\|_{L^q(P_{\frac{r_0}4}(z_0))} \leq C(q, r_0)\epsilon_0.
  \end{equation}
\end{lemma}
\begin{proof} Set $r_1=\frac{r_0}2$.
Then it is easy to see that (\ref{eps00}) also holds for $(u,P,d)$ with $z_0, r_0$
replaced by $z_1, r_1$ for any $z_1\in P_{\frac{r_0}2}(z_0)$.
Applying lemma \ref{le:epsilon} for $(u,P,d)$ on
$P_{r_1}(z_1)$, we conclude that  there exists $\theta_0 \in (0,\frac12)$
such that for any $0<r\le r_1$, it holds that
\begin{equation*}
  \label{Phi:ineq}
  \Phi\left(u, d, P, z_1, \theta_0 r\right) \leq \frac{1}{2}\Phi\left(u, d, P, z_1, r\right)
\end{equation*}
Iterating this inequality $k$-times, $k\ge 1$,  yields
\begin{equation*}
  \label{iter:decay:ineq:1}
  \Phi\left(u, d, P, z_1, \theta_0^kr\right) \leq 2^{-k} \Phi\left(u, d, P, z_1, r\right).
\end{equation*}
It is well known that this implies that there exists $\alpha\in (0,1)$ such that
for any $0<\tau<r\le r_1$, it holds
\beq\label{morrey0}
\displaystyle\Phi\left(u, d, P, z_1, \tau\right)\le (\frac{\tau}{r})^\alpha\Phi\left(u, d, P,  z_1, r \right)
\eeq
holds for any $z_1\in P_{\frac{r_0}2}(z_0)$ and $0<r\le \frac{r_0}2$.

Now we proceed with the Riesz potential estimates of $(u,\nabla d)$ between Morrey spaces as follows.
First, let's recall the notion of Morrey spaces on $\mathbb R^2\times\mathbb R$, equipped with the parabolic metric ${\bf \delta}$:
$$\delta\Big((x,t), (y,s)\Big)=\max\Big\{|x-y|, \sqrt{|t-s|}\Big\}, \ \forall\ (x,t), \ (y,s)\in\mathbb R^2\times\mathbb R.$$
For any open set $U \subset \R^{2+1}$, $1 \leq p < +\infty$, and $0 \leq \lambda \leq 4$,
define the {Morrey Space} $M^{p,\lambda}(U)$ by
\begin{equation}
  M^{p,\lambda}(U) :=
  \left \{
    v\in L^p_{\rm{loc}}(U):
   \left \|v\right\|^p_{M^{p,\lambda}(U)} \equiv \sup_{z \in U, r>0} r^{\lambda - 4} \int_{P_r(z) \cap U} |v|^p < \infty
  \right \}.
  \label{morreyDef}
\end{equation}
It follows from (\ref{morrey0}) that for some $\alpha\in (0,1)$,
\begin{equation}
  u, \ \grad d \in M^{4,4(1-\alpha)} \left (P_{\frac{r_0}2}(z_0) \right), \ (\nabla u,\nabla^2 d, P)\in M^{2, 4-2\alpha}\left(P_{\frac{z_0}2}(z_0)\right).
  \label{morreyugradd}
\end{equation}
Write the equation (\ref{ES2})$_3$ as
\begin{equation}\label{director1}
  \partial_t d -\frac{1}{|\lambda_1|} \lap d = f, \ {\rm{with}}\ \ f:= \left(-u\cdot\nabla d+\Omega \hat d -\frac{\lambda_2}{\lambda_1}A \hat d
+\frac{1}{|\lambda_1|}|\grad d|^2d+\frac{\lambda_2}{\lambda_1}(\hat d^T A \hat d) d\right).
\end{equation}
By \eqref{morreyugradd}, we see that
\begin{equation*}
  f \in M^{2,2(2-\alpha)} \left (P_{\frac{r_0}2}(z_0) \right ).
\end{equation*}
As in  \cite{LW} and  \cite{Huang-Wang},  let $\eta \in C_0^\infty(\R^{2+1})$ be a cut-off function of $P_{\frac{r_0}2}(z_0)$:
$0\le\eta\le 1$, \ $\eta\equiv 1$ in $P_{\frac{r_0}2}(z_0)$, and $|\partial_t\eta|+|\nabla^2\eta|\le Cr_0^{-2}$.
Set $w = \eta^2 d$.  Then we have
\begin{equation}
  \label{Fdef}
  \partial_t w - \frac{1}{|\lambda_1|}\lap w =F, \ \ F:= \eta^2 f + (\partial_t \eta^2 -\frac{1}{|\lambda_1|} \lap \eta^2)(d-d_{z_0,\frac{r_0}2})
- \frac{2}{|\lambda_1|} \grad \eta^2 \cdot \grad d,
\end{equation}
where $d_{z_0, \frac{r_0}2}$ is the average of $d$ over $P_{\frac{r_0}2}(z_0)$.
It is easy to check that $F \in M^{2,2(2-\alpha)}(\R^{2+1})$ and satisfies the estimate
\begin{equation}\label{F-estimate}
  \begin{aligned}
    \Big\| F \Big\|_{M^{2,2(2-\alpha)}(\R^{2+1})} &\leq C\left [ \Phi\left(u, d, P, z_0, r_0\right)+ \|f\|_{M^{2,2(2-\alpha)}(P_{\frac{r_0}2}(z_0))}\right]
\le C\epsilon_0.
  \end{aligned}
\end{equation}
Let $\Gamma(x,t)$ denote the fundamental solution of the heat operator ($\partial_t-\frac{1}{|\lambda_1|}\Delta$) on $\mathbb R^2$. Then
by  the Duhamel formula for (\ref{Fdef}) and  the estimate (see also \cite{Huang-Wang} lemma 3.1):
$$|\nabla\Gamma|(x,t)\lesssim \frac{1}{\delta^{3}((x,t), (0,0))}, \ \forall (x,t)\not=(0,0),$$
we have
\begin{equation}
  \label{riesz:pot:est:1}
  \begin{aligned}
    |\grad w(x,t)|
    &\leq \int_0^t \int_{\R^2} |\grad \Gamma(x-y,t-s)||F(y,s)|
    \leq C \int_{\R^3} \frac{|F(y,s)|}{\delta^3((x,t),(y,s))}:
    = C\mathcal I_1(|F|)(x,t),
  \end{aligned}
\end{equation}
where $\mathcal I_\beta $ is the Riesz potential of order $\beta$ on $\mathbb R^3$ ($\beta\in [0,4]$),
defined by
\begin{equation}
  \label{riesz:pot:def}
  \mathcal{I}_\beta(g) = \int_{\R^{3}} \frac{|g(y,s)|}{\delta((x,t),(y,s))^{4-\beta}}, \ \ \forall\ g \in L^p(\mathbb R^{3}).
\end{equation}
Applying the Riesz potential estimates (see \cite{Huang-Wang} Theorem 3.1),  we conclude that
$\grad w \in M^{\frac{2(2-\alpha)}{1-\alpha}, 2(2-\alpha)}(\R^{3})$ and
\begin{equation}
  \label{morrey:gradw:est}
  \Big\|\grad w \Big\|_{M^{\frac{2(2-\alpha)}{1-\alpha}, 2(2-\alpha)}(\R^{3})}  \lesssim
 \Big\|F\Big\|_{M^{2,2(2-\alpha)}(\R^{3})}
\leq C\epsilon_0.
\end{equation}
Choosing $\alpha\uparrow 1$ and using
$\lim_{\alpha \uparrow 1}\frac{2(2-\alpha)}{1-\alpha} = +\infty$, we can conclude that
for any $1<q<\infty$,
$\nabla w\in L^q(P_{r_0}(z_0))$ and
\begin{equation}
  \label{gradd:L_m:m.gt.1}
  \Big\|\grad w \Big\|_{L^q\left ( P_{r_0}(z_0) \right )}\leq C(q, r_0)\epsilon_0.
\end{equation}
Since $(d-w)$ solves
$$\partial_t(d-w)-\frac{1}{|\lambda_1|}\Delta(d-w)=0 \ {\rm{in}}\ P_{\frac{r_0}2}(z_0),$$
it follows from the standard estimate on the heat equation
that for any $1<q<+\infty$, $\nabla d\in L^q(P_{\frac{r_0}4}(z_0))$ and
\begin{equation}
\label{estimate_d}
  \Big\|\grad d \Big\|_{L^q( P_{\frac{r_0}4}(z_0))}\leq C(q, r_0)\epsilon_0.
\end{equation}

Now we proceed with the estimation of $u$.  Let $v:\mathbb R^2\times [0,+\infty)\to\mathbb R^2$
solve the Stokes equation:
\begin{equation}
\begin{cases}
  \label{aux:stokes}
  \begin{aligned}
    \partial_t v -\frac{\mu_4}2\lap v + \grad Q&=  -\grad \cdot[\eta^2 (\grad d \odot \grad d + u \otimes u)]+\nabla\cdot[\eta^2(\sigma^L(u,d)-\mu_4 A)] &\text{ in } \R^2 \times (0,\infty),\\
    \grad\cdot v &= 0 &\text{ in } \R^2 \times (0,\infty), \\
    v(\cdot,0) &= 0 &\text{ in } \R^2.
  \end{aligned}
\end{cases}
\end{equation}
By using the Oseen kernel (see Leray \cite{Leray}), an estimate for $v$, similar to \eqref{riesz:pot:est:1},  can be given by
\begin{equation}
  \label{riesz:pot:est:2}
  |v(x,t)| \leq C \int_0^t\int_{\R^{2}} \frac{|X(y,s)|}{\delta((x,t),(y,s))^{3}}
\leq C \mathcal I_1(|X|)(x,t),  \ (x,t)\in\mathbb R^2\times (0,+\infty),
\end{equation}
where $X =-\eta^2(\grad d \odot \grad d + u \otimes u )+\eta^2(\sigma^L(u,d)-\mu_4 A)$. As above, we can check
that $X \in M^{2,2(2-\alpha)}(\R^{3})$ and
\begin{equation*}
 \Big\|X\Big\|_{M^{2,2(2-\alpha)}(\R^{3})}
\leq C \left [ \left\||u|+|\grad d|\right\|^2_{M^{4,4-\alpha}(P_{\frac{r_0}2}(z_0))}
+\left\||\nabla u|+|\nabla^2 d|\right\|_{M^{2, 4-2\alpha}(P_{\frac{r_0}2}(z_0))}\right]\le C\epsilon_0.
\end{equation*}
Hence, by  \cite{Huang-Wang} Theorem 3.1, we have that $v \in M^{\frac{2(2-\alpha)}{1-\alpha}, 2(2-\alpha)}(\R^{3})$,  and
\begin{equation}
  \label{morrey:v:est}
   \Big\|v \Big\|_{M^{\frac{2(2-\alpha)}{1-\alpha}, 2(2-\alpha)}(\R^{3})}
\leq C\Big\|X\Big\|_{M^{2,2(2-\alpha)}(\R^{3})}\le C\epsilon_0.
\end{equation}
By sending $\alpha\uparrow 1$, (\ref{morrey:v:est}) implies that for any $1<q<+\infty$,
$v \in L^q\left (P_{r_0}(z_0)\right )$ and
\begin{equation}\label{estimate_q}
\Big\|v \Big\|_{L^q\left (P_{r_0}(z_0)\right )}\le C(q, r_0)\epsilon_0.
\end{equation}
Since $(u-v)$ satisfies the linear homogeneous Stokes equation in $P_{\frac{r_0}2}(z_0)$:
\begin{equation*}
\partial_t (u-v) - \frac{\mu_4}2\lap (u-v) + \grad(P - Q) = 0, \ \grad \cdot (u-v) = 0\ \  \text{ in } \ \ P_{\frac{r_0}2}(z_0).
\end{equation*}
It is well-known that $(u-v) \in L^{\infty}(P_{\frac{r_0}4}(z_0))$.  Therefore we conclude that for any
$1<q<+\infty$, $u\in L^q(P_{\frac{r_0}4}(z_0))$, and
\begin{equation}
  \label{estimate_u}
 \Big\| u \Big\|_{L^q( P_{\frac{r_0}4}(z_0))}\le C(q, r_0) \epsilon_0.
\end{equation}
The estimate (\ref{lq:estimate}) follows from (\ref{estimate_d}) and (\ref{estimate_u}). This completes the proof.
\end{proof}

Now we utilize the integrability estimate (\ref{lq:estimate}) of ($u,\nabla d$) to prove the smoothness of $(u,d)$.\footnote{In fact, we only need to use
$(u, \nabla d)\in L^8$ in the proof for the cases $k=1,2$.}
The argument is based on local inequalities of higher order energy of $(u,\nabla d)$.

\begin{lemma} \label{higher-order-estimate}
Assume that the conditions (\ref{necessary}),
(\ref{parodi}), and (\ref{Leslie_condition}) hold. For any $0<T\le +\infty$ and a bounded domain $O\subset\R^2$, there exists $\epsilon_0>0$ such that
  $(u,P, d)\in L^\infty_t L^2_x\cap L^2_tH^1_x(O\times [0,T])\times L^2(O\times [0,T])\times (L^\infty_tH^1_x\cap L^2_tH^2(O\times [0,T],\mathbb S^2)$
is a suitable weak solution of (\ref{ES2}) , and satisfies, for $z_0=(x_0,t_0)\in O\times (0,T)$ and $P_{r_0}(z_0)\subset O\times (0, T)$,
  \begin{equation}\label{eps0}
   \Phi\left(u, d, P, z_0, r_0\right)
\leq \epsilon_0,
  \end{equation}
 then
$\displaystyle (\nabla^l u,\nabla^{l+1}d) \in \big(L^\infty_tL^2_x\cap L^2_tH^1_x\big)\big(P_{\frac{1+2^{-(l+1)}}2 r_0}(z_0)\big)$
for any $l\ge 0$, and the following estimate holds
  \beq
  \label{k2:estimate}
   &&\sup_{t_0-\frac{((1+2^{-(l+1)})r_0)^2} 4\le t\le t_0}\int_{B_{\frac{1+2^{-(l+1)}}2 r_0}(x_0)}
\left(|\nabla^ l u|^2+|\nabla^{l+1} d|^2\right)\nonumber\\
&&\qquad\qquad\qquad+\int_{P_{\frac{1+2^{-(l+1)}}2 r_0}(z_0)}\left(|\nabla^{l+1} u|^2+|\nabla^{l+2} d|^2
+|\nabla^l P|^2\right)\nonumber\\
&&\qquad\qquad\leq C(l) \epsilon_0.
  \eeq
\end{lemma}
\proof For simplicity, assume $z_0=(0,0)$ and $r_0=2$. We will prove (\ref{k2:estimate}) by an induction on $l\ge 0$. \\
(i) $l=0$: (\ref{k2:estimate}) follows from the local energy inequality (\ref{local_energy_ineq1}), similar to that given by lemma 3.1.\\
(ii) $l\ge 1$:  Suppose that (\ref{k2:estimate}) holds for $l\le k-1$. We want to show (\ref{k2:estimate}) also holds for
$l=k$. From the hypothesis of induction, we have that for all $0\le l\le k-1$,
\beq\label{induction-hypothesis}
&&\sup_{-(1+2^{-(l+1)})^2\le t\le 0}\int_{B_{1+2^{-(l+1)}}}(|\nabla^l u|^2+|\nabla^{l+1}d|^2)\nonumber\\
&&+\int_{P_{1+2^{-(l+1)}}}(|\nabla^{l+1} u|^2+|\nabla^{l+2} d|^2+|\nabla^l P|^2)\le C(l)\epsilon_0.
\eeq
Hence by the  Ladyzhenskaya inequality (\ref{lady}) we have
\beq\label{l4-estimate-ud}
\int_{P_{1+2^{-(l+1)}}}\left(|\nabla^lu|^4+|\nabla^{l+1}d|^4\right)\le C(l)\epsilon_0, \ \forall 0\le l\le k-1.
\eeq
By lemma \ref{high_integrability}, we also have
\beq\label{lq:estimate2}
\|u\|_{L^q(P_{\frac32})}+\|\nabla d\|_{L^q(P_{\frac32})}\le C(q)\epsilon_0, \ \forall 1<q<+\infty.
\eeq
 Take $k$-th order spatial derivative $\nabla^k$ of the equation (\ref{ES2})$_1$, we have\footnote{Strictly speaking we need to take the finite quotient
$D_h^i \nabla^{k-1}$ of (\ref{ES2})$_1$ and then taking limit as $h$ tends to zero.} 
\beq
    \label{kth-equation}
    \partial_t (\nabla^k u)+\nabla^k\nabla\cdot(u\otimes u)+\nabla^k\nabla P=-\nabla^k\nabla\cdot(\nabla d\odot\nabla d)+\nabla^k\nabla\cdot(\sigma^L(u,d)).
\eeq
Let $\eta \in C^\infty_0(B_2)$ such that
$$0\le\eta\le 1,\ \eta\equiv 1 \ {\rm{in}}\  B_{1+2^{-(k+1)}}, \ \eta=0 \ {\rm{outside}} \ B_{1+2^{-k}},
\ |\nabla\eta|+|\nabla^2\eta|\leq 2^{k+4}.$$
Multiplying (\ref{kth-equation}) by $\nabla^k u\eta^2$ and integrating over $B_2$, 
we obtain\footnote{Strictly speaking, we need to multiply the equation by $D_h^i\nabla^{k-1}u \eta^2$.}

\beq
    \label{nabla_u_gronwall}
   \frac{d}{dt} \int_{B_{2}}\frac12|\nabla^k u|^2\eta^2&=&\int_{B_2}\nabla^k(u\otimes u):\nabla(\nabla^ku\eta^2)+\int_{B_2}\nabla^kP \cdot \nabla^k u\cdot\nabla (\eta^2)
\nonumber\\
&&+\int_{B_2}\nabla^k(\nabla d\odot\nabla d):\nabla(\nabla^ku\eta^2)-\int_{B_2} \nabla^k(\sigma^L(u,d)):\nabla(\nabla^k u\eta^2)\nonumber\\
&:=&I_1+I_2+I_3+I_4.
\eeq
We estimate $I_1, I_2, I_3$ as follows.  Applying  H\"older's inequality and the following interpolation inequality:
\beq
    \label{nabel_u_L4}
   \int_{B_{2}}|f|^4\eta^4\lesssim (\int_{B_{2}}|f|^2\eta^2)\left(\int_{B_{2}}|\nabla f|^2\eta^2+\int_{B_{2}}|f|^2|\nabla\eta|^2\right),
 \ \forall\  f\in H^1(\R^2),
\eeq
we have
\beq
    \nonumber
   |I_1|&\lesssim &\int_{B_{2}}[|u||\nabla^ku|+\sum_{j=1}^{k-1} |\nabla^j u||\nabla^{k-j} u|](|\nabla^{k+1}u|\eta^2+|\nabla^k u|\eta|\nabla\eta|)
    \\
\nonumber
  &\leq&
    \left(\delta+C\int_{B_{2}}|\nabla^k u|^2\eta^2\right)\int_{B_{2}}|\nabla^{k+1}u|^2\eta^2+C\sum_{j=0}^{k-1}\int_{\rm{spt}\eta}|\nabla^j u|^4\nonumber\\
    &&+C\left(\int_{\rm{spt}\eta}|\nabla^k u|^2+\int_{\rm{spt}\eta}|\nabla^k u|^2\int_{B_{2}}|\nabla^ku|^2\eta^2\right)
\eeq
where $\delta>0$ is a small constant to be chosen later.
For $I_2$ and $I_3$, we have
\beq
|I_2|&\lesssim& \int_{B_2}|\nabla^{k-1}P|(|\nabla^{k+1} u|\eta|\nabla\eta|+|\nabla^k u||\nabla^2(\eta^2)|)\nonumber\\
&\leq& \delta\int_{B_2}|\nabla^{k+1} u|^2\eta^2+C\int_{\rm{spt}\eta}(|\nabla^{k-1}P|^2+|\nabla^k u|^2).
\eeq
\beq
    \nonumber
    |I_3|&\lesssim &\int_{B_{2}}(|\nabla d||\nabla^{k+1}d|+\sum_{j=1}^{k-1}|\nabla^{j+1}d||\nabla^{k+1-j}d|)(|\nabla^{k+1} u|\eta^2+|\nabla^ku||\nabla(\eta^2)|)
    \\
&\leq&\left(\delta+C\int_{B_2}|\nabla^{k+1}d|^2\eta^2\right)\int_{B_{2}}(|\nabla^{k+1}u|^2+|\nabla^{k+2} d|^2)\eta^2
+C\int_{\rm{spt}\eta}(|\nabla u|^2+\sum_{j=1}^k|\nabla^j d|^4)
\nonumber\\
&&+C\left(\int_{\rm{spt}\eta}(|\nabla^k u|^2+|\nabla^{k+1} d|^2)+\int_{\rm{spt}\eta}|\nabla^{k+1}d|^2\int_{B_2}|\nabla^{k+1} d|^2\eta^2\right).
\eeq
For $I_4$, we need to proceed as follows. Set
$$\sigma^L_k(u, d)
=\mu_1(\hat d\otimes \hat d: \nabla^kA)\hat d\otimes \hat d+\mu_2 \nabla^k\hat N\otimes \hat d +\mu_3 \hat d\otimes \nabla^k\hat N
+\mu_4 \nabla^k A+\mu_5(\nabla^kA\cdot\hat d)\otimes \hat d+\mu_6 \hat d\otimes(\nabla^kA\cdot\hat d),$$
and
$$\omega^L_k(u, d):=\nabla^k(\sigma^L(u,d))-\sigma^L_k(u,d).$$
Then we have
\beq\label{highest_order}
-\int_{B_2}\nabla^k(\sigma^L(u,d)):\nabla(\nabla^k u\eta^2)
&=&-\int_{B_2}\sigma^L_k(u,d):(\nabla^{k+1} u)\eta^2-\int_{B_2} \sigma^L_k(u,d): \nabla^k u\otimes\nabla (\eta^2)\nonumber\\
&&-\int_{B_2}\omega_k^L(u,d):(\nabla^{k+1}u \eta^2+\nabla^k u\otimes \nabla(\eta^2))\nonumber\\
&:=&J_1+J_2+J_3.
\eeq
To estimate $J_1, J_2, J_3$, we take $\nabla^k$ of the equation (\ref{ES2})$_3$ to get\footnote{Strictly speaking, we need to take $D_h^i\nabla^{k-1}$ of the equation.}
\beq\label{kth-ES2-3}
\nabla^k N +\frac{\lambda_2}{\lambda_1} \nabla^k(A \hat d)=\frac{1}{|\lambda_1|} \left(\Delta \nabla^k d+\nabla^k(|\nabla d|^2 d)\right)
+\frac{\lambda_2}{\lambda_1}\nabla^k\left((\hat d^T A\hat d)d\right).
\eeq
Since $|d|=1$, we have that $|\nabla d|^2=-\langle \Delta d, d\rangle$. Denote by $\#$ the multi-linear map with constant coefficients.
It is well known (see \cite{BKM} ) that for any $l\ge 0$,
$$\|\nabla^l(d\#d)\|_{L^2}\lesssim \|d\|_{L^\infty}\|\nabla^l d\|_{L^2}\lesssim \|\nabla^l d\|_{L^2},$$
and
$$\|\nabla^l(d\#d\#d)\|_{L^2}\lesssim \|d\|_{L^\infty}\|\nabla^l (d\#d)\|_{L^2}+\|d\#d\|_{L^\infty}\|\nabla^l d\|_{L^2}\lesssim \|\nabla^l d\|_{L^2}.$$
Therefore by (\ref{induction-hypothesis}) and (\ref{l4-estimate-ud}) we have that for any $0\le l\le k-1$,
$$\nabla(d\#d),\nabla(d\#d\#d)\in \Big(L^\infty_t H^{l}_x\cap L^2_tH^{l+1}_x\cap L^4\Big)\Big(P_{1+2^{-(l+1)}}\Big).$$
and
\beq\label{double-d-estimate}
\Big\||\nabla(d\#d)|+|\nabla(d\#d\#d)|\Big\|_{L^4\cap L^\infty_t H^{l}_x\cap L^2_tH^{l+1}_x(P_{1+2^{-(l+1)}})}\le C(l)\epsilon_0.
\eeq
The estimate (\ref{double-d-estimate}) also holds for $d\#d\#d\#d$.
Applying the equation (\ref{kth-ES2-3}) we have
\beq
&&|J_2|\lesssim \int_{B_2}(|\nabla^{k+1}u|+|\nabla^k N|)|\nabla^k u|\eta|\nabla\eta|\nonumber\\
&&\lesssim
\int_{B_2}\Big[|\nabla^{k+1}u|+|\nabla^{k+2}d|+\sum_{j=1}^k |\nabla^j u||\nabla^{k-j+1}d|\Big]|\nabla^k u|\eta|\nabla\eta|\nonumber\\
&&+\Big[\sum_{l=0}^{k-1} |\nabla^{l+2}d| |\nabla^{k-l}(d\#d)|\Big]|\nabla^k u|\eta|\nabla\eta|\nonumber\\
&&+\Big[\sum_{l=0}^{k-1}|\nabla^{k-l}(d\#d\#d)||\nabla^{l+1} u|\Big]|\nabla^ku|\eta|\nabla\eta|:=J_{21}+J_{22}+J_{23}.\nonumber
\eeq
Direct calculations imply
\beq\nonumber
J_{21}&\le& \left[\delta+C\int_{B_{2}}(|\nabla^k u|^2+|\nabla^{k+1}d|^2)\eta^2\right]\int_{B_{2}}(|\nabla^{k+1}u|^2+|\nabla^{k+2} d|^2)\eta^2\\
&+&C\int_{\rm{spt}\eta}(|\nabla^k u|^2+|\nabla^{k+1} d|^2)\int_{B_2}(|\nabla^k u|^2+|\nabla^{k+1} d|^2)\eta^2+C\sum_{i=1}^k \int_{\rm{spt}\eta}(|\nabla^id|^4+|\nabla^i u|^2),\nonumber
\eeq
\beq
J_{22}&\le& \int_{B_2}\Big[|\nabla^k d||\nabla d|^2+|\nabla ^{k+1}d||\nabla d|+\sum_{l=0}^{k-2}|\nabla^{l+2}d||\nabla^{k-l}(d\#d)|\Big]|\nabla^k u|\eta|\nabla\eta|\nonumber\\
&\le& C\int_{B_{2}}|\nabla^k u|^2\eta^2\int_{B_{2}}|\nabla^{k+1}u|^2\eta^2+C\int_{\rm{spt}\eta}|\nabla^k u|^2\int_{B_2}|\nabla^k u|^2\eta^2\nonumber\\
&&+C\int_{\rm{spt}\eta}\Big[|\nabla d|^8+|\nabla^k d|^2+|\nabla^{k+1}d|^2+\sum_{i=1}^k |\nabla^id|^4\Big],\nonumber
\eeq
and
\beq\nonumber
J_{23}&\le& \int_{B_2}\Big[|\nabla^k u||\nabla d|+\sum_{l=1}^{k-1}|\nabla^{k+1-l} (d\#d\#d)||\nabla^{l} u|\Big]|\nabla^k u|\eta|\nabla\eta|
\\
&\lesssim&C\int_{B_{2}}(|\nabla^k u|^2+|\nabla^{k+1}d|^2)\eta^2\int_{B_{2}}(|\nabla^{k+1}u|^2+|\nabla^{k+2} d|^2)\eta^2\nonumber\\
&+&C\int_{\rm{spt}\eta}(|\nabla d|^8+|\nabla^k u|^2+|\nabla^{k+1}d|^2)
+C\sum_{j=0}^{k-1}\int_{\rm{spt}\eta}(|\nabla^j u|^4+|\nabla^{j+1}d|^4)\nonumber\\
&+&
C\int_{\rm{spt}\eta}(|\nabla^k u|^2+|\nabla^{k+1} d|^2)\int_{B_2}(|\nabla^k u|^2+|\nabla^{k+1} d|^2)\eta^2.\nonumber
\eeq
By the definition of $\omega^L_k(u,d)$ and the equation (\ref{ES2})$_3$, we have
\beq\label{omega-k}
&&|\omega^L_k(u,d)|\lesssim\sum_{l=1}^{k}|\nabla^l u|(|\nabla^{k+1-l}(d\#d)|+|\nabla^{k+1-l}(d\#d\#d)|)
 +\sum_{l=0}^{k-1}|\nabla^l\hat N||\nabla^{k-l} d|\nonumber\\
&&\lesssim \sum_{l=1}^{k}|\nabla^l u|(|\nabla^{k+1-l}(d\#d)|+|\nabla^{k+1-l}(d\#d\#d)|)+|\nabla d||\nabla^{k+1}d|
+\sum_{l=2}^{k}|\nabla^{l}d|^2\nonumber\\
&&\quad+\sum_{l=0}^{k-1}\Big[|\nabla^l (Ad)|+|\nabla^l(\nabla^2 d\#d\#d)|
+|\nabla^l(A\#d\#d\#d)\Big]|\nabla^{k-l}d|\nonumber\\
&&\lesssim (|\nabla^k u|+|\nabla^{k+1}d|)(|\nabla d|+|\nabla(d\#d)|+|\nabla(d\#d\#d)|)
\nonumber\\
&&\quad+\sum_{l=1}^k (|\nabla^{l-1}u|^2+|\nabla^l d|^2+|\nabla^l(d\#d)|^2+|\nabla^l(d\#d\#d)|^2).
\eeq
Hence we can estimate
\beq
|J_3|&\lesssim&\int_{B_2}|\omega_k^L(u,d)|(|\nabla^{k+1} u|\eta^2+|\nabla^k u|\eta|\nabla\eta|)\nonumber\\
&\leq &
\left[\delta+C\int_{B_2}(|\nabla^k u|^2+|\nabla^{k+1}d|^2)\eta^2\right] \int_{B_2}(|\nabla^{k+1} u|^2+|\nabla^{k+2} d|^2)\eta^2
+C\int_{\rm{spt}\eta}(|\nabla^k u|^2+|\nabla^{k+1} d|^2)\nonumber\\
&&+C\sum_{l=1}^k \int_{\rm{spt}\eta}(|\nabla^{l-1}u|^4+|\nabla^l d|^4+|\nabla^l(d\#d)|^4+|\nabla^l(d\#d\#d)|^4)\nonumber\\
&&+C\int_{\rm{spt}\eta}(|\nabla^k u|^2+|\nabla^{k+1}d|^2)\int_{B_2}(|\nabla^k u|^2+|\nabla^{k+1} d|^2)\eta^2.
\eeq
The most difficult term to handle is $J_1$, since the integrands involve terms consisting of the highest order factors
$\nabla^{k+1}u$ and $\nabla^{k+2}d$. Here we need to apply (\ref{ES2}) to cancel some of those terms and employ
the condition (\ref{Leslie_condition}) to argue that other terms are non-positive. For this, we proceed as follows
(similar to lemma 3.2).
\beq\label{J1-estimate}
    \nonumber
 J_1&=&-\int_{B_2}\Big[\mu_1(\hat d\otimes \hat d: \nabla^kA)\hat d\otimes \hat d+\mu_2 \nabla^k\hat N\otimes \hat d +\mu_3 \hat d\otimes \nabla^k\hat N\nonumber\\
&&\qquad\ \ \ +\mu_4 \nabla^kA+\mu_5(\nabla^kA \hat d)\otimes \hat d+\mu_6 \hat d\otimes(\nabla^kA\hat d)\Big]:\Big[\nabla^kA\eta^2+\nabla^k\Omega\eta^2\Big]\nonumber\\
&=&-\int_{B_2}\mu_1 |\hat d^T \nabla^kA\hat d|^2\eta^2+\mu_4 |\nabla^kA|^2\eta^2+(\mu_5+\mu_6)|\nabla^kA\hat d|^2\eta^2 \nonumber\\
&&\qquad\ \ +\lambda_1  \nabla^k\hat N\cdot(\nabla^k\Omega\hat d)\eta^2-\lambda_2\nabla^k\hat N\cdot(\nabla^kA\hat d)\eta^2+\lambda_2(\nabla^kA\hat d)(\nabla^k\Omega\hat d)\eta^2.
\eeq
Applying the equation (\ref{kth-ES2-3}), we obtain
\beq\label{n_hat1}
&&\lambda_1 \nabla^k\hat N\cdot(\nabla^k\Omega\hat d)+\lambda_2 (\nabla^kA\hat d)(\nabla^k\Omega\hat d)
=-\langle \nabla^k\Omega\hat d, \Delta \nabla^k\hat d\rangle
-\langle \nabla^k(|\nabla d|^2 \hat d), \nabla^k\Omega \hat d\rangle\nonumber\\
&&\qquad\qquad\qquad\qquad+\lambda_2\Big \langle \nabla^k((\hat d^T A\hat d)\hat d), \nabla^k\Omega\hat d\Big\rangle
-\lambda_2\Big\langle \nabla^k(A\hat d)-\nabla^k A \hat d, \nabla^k \Omega\hat d\Big\rangle,
\eeq
and
\beq\label{n_hat2}
-\lambda_2\nabla^k\hat N(\nabla^kA\hat d)&=&\frac{\lambda_2}{\lambda_1}\Big\langle \Delta \nabla^k\hat d, \nabla^kA\hat d\Big\rangle
-\frac{\lambda_2^2}{\lambda_1}|\hat d^T \nabla^kA d|^2+\frac{\lambda_2^2}{\lambda_1}|\nabla^kA\hat d|^2\nonumber\\
&&+\frac{\lambda_2^2}{\lambda_1}
\Big\langle \nabla^k(A \hat d)-\nabla^kA\hat d,  \nabla^k A\hat d\Big\rangle+\frac{\lambda_2}{\lambda_1}\Big\langle \nabla^k(|\nabla d|^2 \hat d), \nabla^k A\hat d\Big\rangle\nonumber\\
&&+\frac{\lambda_2^2}{\lambda_1}\Big\langle(\hat d^T \nabla^k A\hat)\hat d- \nabla^k((\hat d^T A \hat d)\hat d), \nabla^k A\hat d\Big\rangle.
\eeq
Putting (\ref{n_hat1}) and (\ref{n_hat2}) into (\ref{J1-estimate}) yields
\beq\label{J1-estimat1}
&&J_1\nonumber\\
&=&-\int_{B_2}[(\mu_1-\frac{\lambda_2^2}{\lambda_1})|\hat d^T \nabla^kA\hat d|^2+\mu_4 |\nabla^kA|^2+(\mu_5+\mu_6+\frac{\lambda_2^2}{\lambda_1})|\nabla^kA\hat d|^2+\langle \frac{\lambda_2}{\lambda_1} \nabla^kA\hat d-\nabla^k\Omega\hat d, \Delta \nabla^k\hat d\rangle]\eta^2\nonumber\\
&-&\int_{B_2}\Big[\Big\langle\lambda_2\nabla^k ((\hat d^T A\hat d)\hat d)-\nabla^k(|\nabla d|^2\hat d)
-\lambda_2\left(\nabla^k(A\hat d)-\nabla^k A\hat d\right), \nabla^k\Omega\hat d\Big\rangle
+\frac{\lambda_2}{\lambda_1}\Big\langle \nabla^k(|\nabla d|^2 \hat d), \nabla^k A\hat d\Big\rangle\nonumber\\
&+&\frac{\lambda_2^2}{\lambda_1}\Big\langle \nabla^k(A \hat d)-\nabla^k A\hat d,  \nabla^k A\hat d\Big\rangle
+\frac{\lambda_2^2}{\lambda_1}\Big\langle (\hat d^T \nabla^k A\hat d)\hat d-\nabla^k((\hat d^T A \hat d)\hat d), \nabla^k A\hat d\Big\rangle\Big]\eta^2.
\eeq
Multiplying the equation (\ref{kth-ES2-3}) by $\Delta \nabla^kd\eta^2$ and integrating over $B_{2}$, we have\footnote{Strictly speaking, we need to
multiply the equation by $\Delta D_h^i\nabla^{k-1}d \eta^2$.}
\beq
    \label{nabla^2d_gronwall}
   && \frac{d}{dt}\int_{B_2}\frac12|\nabla^{k+1}d|^2\eta^2+\frac{1}{|\lambda_1|}\int_{B_2}|\Delta \nabla^k d|^2\eta^2=\int_{B_2} \Big\langle\frac{\lambda_2}{\lambda_1}\nabla^kA\hat d-\nabla^k\Omega\hat d, \Delta \nabla^k\hat d\Big\rangle\eta^2\nonumber\\
&& +\int_{B_2}\Big\langle \nabla^k(u\cdot \nabla d)+(\nabla^k\Omega \hat d-\nabla^k(\Omega\hat d))+\frac{\lambda_2}{\lambda_1} (\nabla^k(A\hat d)-\nabla^k A\hat d)
+\frac{1}{\lambda_1}\nabla^k(|\nabla d|^2 d), \Delta \nabla^kd\Big\rangle \eta^2\nonumber\\
&&-\int_{B_2}\partial_t \nabla^kd\cdot\nabla \nabla^kd\cdot\nabla \eta^2
-\frac{\lambda_2}{\lambda_1}\int_{B_2}\Big\langle \nabla^k((\hat d^T A \hat d)d),\Delta \nabla^kd\Big\rangle\eta^2.
\eeq
Adding (\ref{nabla_u_gronwall}) and (\ref{nabla^2d_gronwall}), using  (\ref{Leslie_condition}), we obtain
\beq\label{2nd_order_inequality}
&&\frac{d}{dt} \int_{B_{2}}(|\nabla^k u|^2+|\nabla^{k+1} d|^2)\eta^2
+\int_{B_2}\left(2\mu_4|\nabla^k A|^2+\frac{2}{|\lambda_1|}|\nabla^{k+2} d|^2\right)\eta^2
\le I_1+I_2+I_3+J_2+J_3\nonumber\\
&&+\int_{B_2}\Big\langle \nabla^k(u\cdot \nabla d)+(\nabla^k\Omega \hat d-\nabla^k(\Omega\hat d))+\frac{\lambda_2}{\lambda_1}
(\nabla^k (A\hat d)-\nabla^k A\hat d), \Delta \nabla^kd\Big\rangle \eta^2\nonumber\\
&&-\int_{B_2}\partial_t \nabla^kd\cdot\nabla \nabla^kd\cdot\nabla \eta^2
+\int_{B_2}\Big\langle \frac{1}{\lambda_1}\nabla^k(|\nabla d|^2 d)-\frac{\lambda_2}{\lambda_1}\nabla^k((\hat d^T A \hat d)d),\Delta \nabla^kd\Big\rangle\eta^2\nonumber\\
&&-\int_{B_2}\Big[\Big\langle\lambda_2\nabla^k ((\hat d^T A\hat d)\hat d)-\nabla^k(|\nabla d|^2\hat d)
-\lambda_2\left(\nabla^k(A\hat d)-\nabla^k A\hat d\right), \nabla^k\Omega\hat d\Big\rangle
+\frac{\lambda_2}{\lambda_1}\Big\langle \nabla^k(|\nabla d|^2 \hat d), \nabla^k A\hat d\Big\rangle\nonumber\\
&&+\frac{\lambda_2^2}{\lambda_1}\Big\langle \nabla^kA \hat d-\nabla^k(A\hat d),  \nabla^k A\hat d\Big\rangle
-\frac{\lambda_2^2}{\lambda_1}\Big\langle \nabla^k((\hat d^T A \hat d)\hat d)-(\hat d^T \nabla^k A\hat d)\hat d, \nabla^k A\hat d\Big\rangle\Big]\eta^2
\nonumber\\
&:=&I_1+I_2+I_3+J_2+J_3+K_1+K_2+K_3+K_4.
\eeq
The terms $K_1, K_2, K_3, K_4$ can be estimated as follows.
\beq\label{K1-estimate}
&&|K_1|\lesssim\int_{B_2}\big[|\nabla d||\nabla^k u|+|u||\nabla^{k+1}d|+\sum_{l=1}^{k-1}(|\nabla^l u|^2+|\nabla^{l+1}d|^2)\big] |\nabla^{k+2}d|\eta^2\nonumber\\
&&\leq \big[\delta+C\int_{B_2}(|\nabla^k u|^2+|\nabla^{k+1} d|^2)\eta^2\big]\int_{B_2}(|\nabla^{k+1} u|^2+|\nabla^{k+2} d|^2)\eta^2
+C\sum_{l=0}^{k-1}\int_{\rm{spt}\eta}(|\nabla^l u|^2+|\nabla^{l+1}d|^2)\nonumber\\
&&+C\int_{\rm{spt}\eta}(|\nabla^ku|^2+|\nabla^{k+1} d|^2)
\int_{B_2}(|\nabla^k u|^2+|\nabla^{k+1} d|^2)\eta^2,
\eeq
For $K_2$, by the equation (\ref{kth-ES2-3}) and the fact $|\nabla d|^2=-\langle d,\Delta d\rangle$ we have
$$|\partial_t \nabla^k d|\lesssim |\nabla^k(u\cdot\nabla d)+|\nabla^k(\Omega\hat d)|+|\nabla^k(A\hat d)|+|\nabla^{k+2}d|
+|\nabla^k((\hat d^T A\hat d)d)|+|\nabla^k(\langle d,\Delta d\rangle d)|,
$$
so that
\beq
\label{K2-estimate}
| K_2|&\lesssim&\int_{B_2}\big[|\nabla^{k+2}d|+|\nabla^{k+1}u|+|\nabla^{k+1}d|(|u|+|\nabla d|)+|\nabla^k u|(|\nabla d|+|\nabla(d\#d\#d)|)\nonumber\\
&&+\sum_{l=2}^k (|\nabla^{l-1}u|^2+|\nabla^ld |^2+|\nabla^l(d\#d)|^2+|\nabla^l(d\#d\#d)|^2)\big]|\nabla^{k+1}d|\eta|\nabla\eta|\nonumber\\
&\le&\big[\delta+C\int_{B_2}(|\nabla^k u|^2+|\nabla^{k+1}d|^2)\eta^2\big]\int_{B_2}(|\nabla^{k+1}u|^2+|\nabla^{k+2} d|^2)\eta^2\nonumber\\
&&+\sum_{l=1}^k\int_{\rm{spt}\eta}(|\nabla^{l-1} u|^4+|\nabla^l d|^4+|\nabla^l(d\#d)|^4+|\nabla^l(d\#d\#d)|^4)\\
&+&C\int_{\rm{spt}\eta}(|\nabla^k u|^2+|\nabla^{k+1}d|^2)
+C\int_{\rm{spt}\eta}(|\nabla^k u|^2+|\nabla^{k+1}d|^2)\int_{B_2}(|\nabla^{k}u|^2+|\nabla^{k+1} d|^2)\eta^2.
\nonumber
\eeq
For $K_4$, we first estimate the terms inside the integrand. Since
$$\langle\nabla^k(\hat d^TA\hat d) \hat d, \nabla^k\Omega\hat d\rangle =0,$$
it follows
\beq\nonumber
&&|\langle\nabla^k((\hat d^TA\hat d) \hat d), \nabla^k\Omega\hat d\rangle|
=|\langle\nabla^k((\hat d^TA\hat d) \hat d)-\nabla^k(\hat d^TA\hat d) \hat d, \nabla^k\Omega\hat d\rangle|\nonumber\\
&&\lesssim \big(|\nabla^k u||\nabla d|+\sum_{l=0}^{k-2}|\nabla^{l+1}u||\nabla^{k-l}(d\#d\#d)|\big)|\nabla^{k+1}u|\nonumber\\
&&\lesssim \big[|\nabla^k u||\nabla d|+\sum_{l=1}^{k-1}(|\nabla^{l}u|^2+|\nabla^{l+1}(d\#d\#d)|^2)\big]|\nabla^{k+1}u|.\nonumber
\eeq
We also have
\beq\nonumber
&&|\langle\nabla^k(|\nabla d|^2 d), \nabla\Omega \hat d\rangle|
=|\langle\nabla^k(|\nabla d|^2 d)-\nabla^k(|\nabla d|^2)\hat d, \nabla^k\Omega \hat d\rangle|\nonumber\\
&&\lesssim [|\nabla^k d||\nabla d|^2+\sum_{l=0}^{k-2}\nabla^l(|\nabla d|^2)|\nabla^{k-l}d|]|\nabla^{k+1}u|\nonumber\\
&&\lesssim \big[|\nabla^k d||\nabla d|^2+\sum_{l=2}^{k}(|\nabla^ld|^2+|\nabla^l(d\#d)|^2\big]|\nabla^{k+1}u|,\nonumber
\eeq
\beq\nonumber
&&|\langle\nabla^k(A\hat d)-\nabla^k A\hat d, \nabla^k\Omega \hat d\rangle|
\lesssim \big[|\nabla^k u||\nabla d|+\sum_{l=1}^{k-1}(|\nabla^l u|^2+|\nabla^{l+1}d|^2)\big]|\nabla^{k+1}u|,\nonumber
\eeq
\beq\nonumber
&&|\langle\nabla^k(|\nabla d|^2 d),\nabla^k A\hat d\rangle|
\lesssim \big[|\nabla^{k+1}d||\nabla d|+|\nabla^k d||\nabla d|^2+\sum_{l=2}^k (|\nabla^l d|^2+|\nabla^l(d\#d)|^2)\big]|\nabla^{k+1}u|,
\nonumber
\eeq
and
\beq\nonumber
|\big\langle \nabla^kA \hat d-\nabla^k(A\hat d),  \nabla^k A\hat d\big\rangle|
\lesssim \big[|\nabla^k u||\nabla d|+\sum_{l=1}^{k-1}(|\nabla^l u|^2+|\nabla^{l+1}d|^2)\big]|\nabla^{k+1}u|,
\eeq
\beq
&&|\big\langle \nabla^k((\hat d^T A \hat d)\hat d)-(\hat d^T \nabla^k A\hat d)\hat d, \nabla^k A\hat d\big\rangle|\nonumber\\
&&\lesssim \big[|\nabla^k u||\nabla(d\#d\#d)|+\sum_{l=1}^{k-1}(|\nabla^l u|^2+|\nabla^{l+1}(d\#d\#d)|^2)\big]|\nabla^{k+1}u|.\nonumber
\eeq
Putting all these estimates together, we would have
\beq
\label{K4-estimate}
&&|K_4|\lesssim  \int_{B_2}\big[(|\nabla^k u|+|\nabla^{k+1}d|)|\nabla d|+|\nabla^k d||\nabla d|^2+|\nabla^k u||\nabla(d\#d\#d)|\big]|\nabla^{k+1}u|\eta^2
\nonumber\\
&&+\int_{B_2}\big[\sum_{l=1}^{k-1}(|\nabla^l u|^2+|\nabla^{l+1}d|^2+|\nabla^{l+1}(d\#d)|^2+|\nabla^{l+1}(d\#d\#d)|^2)\big]|\nabla^{k+1}u|\eta^2\nonumber\\
&&\leq \big[\delta+C\int_{B_2}(|\nabla^k u|^2+|\nabla^{k+1}d|^2)\eta^2\big]\int_{B_2}(|\nabla^{k+1} u|^2+|\nabla^{k+2}d|^2)\eta^2\nonumber\\
&&+C\sum_{l=0}^{k-1} \int_{\rm{spt}\eta}(|\nabla^l u|^4+|\nabla^{l+1}d|^4+|\nabla^{l+1}(d\#d)|^4+|\nabla^{l+1}(d\#d\#d)|^4)\nonumber\\
&&+C\int_{\rm{spt}\eta}(|\nabla^k u|^2+|\nabla^{k+1}d|^2)\int_{B_2}(|\nabla^{k} u|^2+|\nabla^{k+1}d|^2)\eta^2 +C\int_{\rm{spt}\eta}|\nabla d|^8.
\eeq
To estimate $K_3$, we first estimate both terms inside the integrand.
\beq
|\big\langle\nabla^k(|\nabla d|^2 d),\Delta \nabla^kd\big\rangle|\nonumber
\lesssim \big[|\nabla^k d||\nabla d|^2+|\nabla^{k+1} d||\nabla d|+\sum_{l=2}^k(|\nabla^l d|^2+|\nabla^l(d\#d)|^2)\big]|\nabla^{k+2}d|.
\eeq
Since $|d|=1$, it follows $\langle d,\Delta d\rangle=-|\nabla d|^2$ and
$$\langle d,\Delta \nabla^k d\rangle
=-\nabla^k(|\nabla d|^2)-\sum_{j=0}^{k-1}\langle\Delta\nabla^j d,\nabla^{k-j}d\rangle.$$
Therefore we have
\beq
\nonumber
&&\Big|\big\langle\nabla^k((\hat d^T A \hat d)d),\Delta \nabla^kd\big\rangle\Big|\nonumber\\
&&=\Big|(\hat d^T\nabla^k A\hat d)\langle d,\Delta \nabla^k d\rangle
+\big\langle(\sum_{j=0}^{k-1}\nabla^j A\nabla^{k-j}(d\#d))d, \Delta\nabla^k d\big\rangle
+\big\langle(\sum_{j=0}^{k-1}\nabla^j (\hat d^T A \hat d)\nabla^{k-j})d, \Delta\nabla^k d\big\rangle\Big|\nonumber\\
&&=\Big|-(\hat d^T\nabla^k A\hat d)\big[\nabla^k(|\nabla d|^2)+\sum_{j=0}^{k-1}\langle\Delta\nabla^j d,\nabla^{k-j}d\rangle\big]
+\big\langle\sum_{j=0}^{k-1}\nabla^j A\nabla^{k-j}(d\#d))d, \Delta\nabla^k d\big\rangle\nonumber\\
&&\quad+\big\langle\big(\sum_{j=0}^{k-1}\nabla^j (\hat d^T A \hat d)\nabla^{k-j}d, \Delta\nabla^k d\big\rangle\Big|\nonumber\\
&&\lesssim (|\nabla^{k+1}d||\nabla d|+|\nabla^k d||\nabla d|^2)|\nabla^{k+1}u|
+|\nabla^k u|(|\nabla(d\#d)|+|\nabla(d\#d\#d)|)|\nabla^{k+2}d|\nonumber\\
&&\quad+\big(\sum_{l=2}^{k}|\nabla^l d|^2\big)|\nabla^{k+1}u|+\big[\sum_{l=1}^{k-1}(|\nabla^l u|^2+|\nabla^{l+1}(d\#d)|^2+|\nabla^{l+1}(d\#d\#d)|^2)\big]
|\nabla^{k+2} d|.\nonumber
\eeq
Substituting these two estimates into $K_3$, we would obtain
\beq\label{K3-estimate}
&&|K_3|\nonumber\\
&&\lesssim \int_{B_2}\big[|\nabla^{k+1}d||\nabla d|+|\nabla^k d||\nabla d|^2+|\nabla^k u|(|\nabla(d\#d)|+|\nabla(d\#d\#d)|)\big]
(|\nabla^{k+1}u|+|\nabla^{k+2}d|)\eta^2\nonumber\\
&&+\int_{B_2}\big[\sum_{l=1}^{k-1}(|\nabla^l u|^2+|\nabla^ld|^2+|\nabla^{l+1}(d\#d)|^2+|\nabla^{l+1}(d\#d\#d)|^2)\big]
(|\nabla^{k+1}u|+|\nabla^{k+2} d|)\eta^2\nonumber\\
&&\le \big[\delta+\int_{B_2}(|\nabla^k u|^2+|\nabla^{k+1}d|^2)\eta^2\big]
\int_{B_2}(|\nabla^{k+1}u|^2+|\nabla^{k+2}d|^2)\eta^2\nonumber\\
&&+C\int_{\rm{spt}\eta}\big[|\nabla d|^4+|\nabla d|^8+\sum_{l=1}^{k-1}(|\nabla^l u|^4+|\nabla^{l+1}d|^4+|\nabla^{l+1}(d\#d)|^4
+|\nabla^{l+1}(d\#d\#d)|^4)\big]\nonumber\\
&&+C\int_{\rm{spt}\eta}(|\nabla^k u|^2+|\nabla^{k+1}d|^2)
\int_{B_2}(|\nabla^{k}u|^2+|\nabla^{k+1}d|^2)\eta^2.
\eeq
Finally, by substituting all these estimates on $I_i$'s, $J_i$'s, $K_i$'s, and $L_i$'s into (\ref{2nd_order_inequality}) we obtain
\beq\label{2nd_order_inequality1}
&&\frac{d}{dt} \int_{B_{2}}(|\nabla^k u|^2+|\nabla^{k+1} d|^2)\eta^2
+\int_{B_2}\left(\mu_4|\nabla^{k+1} u|^2+\frac{2}{|\lambda_1|}|\nabla^{k+2} d|^2\right)\eta^2\nonumber\\
&&\le\Big[\delta+C\int_{B_2}(|\nabla ^ku|^2+|\nabla^{k+1} d|^2)\eta^2\Big]\int_{B_2}(|\nabla^{k+1} u|^2+|\nabla^{k+2} d|^2)\eta^2\nonumber\\
&&\quad+ C\int_{\rm{spt}\eta}(|\nabla^k u|^2+|\nabla^{k+1} d|^2)\int_{B_2}(|\nabla^k u|^2+|\nabla^{k+1} d|^2)\eta^2\nonumber\\
&&\quad+C\int_{\rm{spt}\eta}\big[|u|^8+|\nabla d|^8+\sum_{l=0}^{k-1}(|\nabla^l u|^4+|\nabla^{l+1} d|^4+|\nabla^{l+1} (d\#d)|^4
+|\nabla^{l+1}(d\#d\#d)|^4)\big]\nonumber\\
&&\quad+C\int_{\rm{spt}\eta}\big[\sum_{l=1}^k(|\nabla^l u|^2+|\nabla^{l+1} d|^2)+|\nabla^{l-1}P|^2\big].
\eeq
By Fubini's theorem, we can choose $t_*\in [-4, -\frac94]$ such that
$$\int_{B_2}(|\nabla^k u|^2+|\nabla^{k+1}d|^2)\eta^2(t_*)\le 8\int_{P_2}(|\nabla^ku|^2+|\nabla^{k+1} d|^2)\eta^2\le 8\epsilon_0.$$
For a large constant $C_1>8$ to be chosen later, define $T_*\in [t_*, 0]$  by
$$T_*=\sup_{t_*\le t\le 0} \Big\{\int_{B_2}(|\nabla^k u|^2+|\nabla^{k+1} d|^2)\eta^2(s)<C_1\epsilon_0, \ \forall \ t_*\le s\le t\ \Big\}.$$
\noindent{\it Claim 1}.  If $\epsilon_0>0$ is sufficiently small, then $T_*=0$.\\
First, by continuity we know that $T_*>t_*$. Suppose that $T_*<0$. Then we have
\beq\label{max_time}
\int_{B_2}(|\nabla^k u|^2+|\nabla^{k+1} d|^2)\eta^2(s)<C_1\epsilon_0\ \ \forall  \ t_*\le s<T_*;
\  \int_{B_2}(|\nabla^k u|^2+|\nabla^{k+1} d|^2)\eta^2(T_*)=C_1\epsilon_0.
\eeq
By choosing sufficiently small $\epsilon_0>0$ and $\delta>0$, we may assume
$$\Big[\delta+C\int_{B_2}(|\nabla^k u|^2+|\nabla^{k+1}d|^2)\eta^2(t)\Big]\le\delta+CC_1\epsilon_0\le \frac12\min\left\{\mu_4, \frac{2}{|\lambda_1|}\right\}, \ \forall \ t_*\le t\le T_*.$$
Set
$$\phi(t):=\int_{t_*}^t\int_{\rm{spt}\eta}(|\nabla^k u|^2+|\nabla^{k+1} d|^2)(s)\,ds, \ \forall t_*\le t\le T_*.$$
Then by integrating (\ref{2nd_order_inequality1}) over $t\in [t_*, T_*]$ and applying Gronwall's inequality, we have
\beq
\label{2nd_order_inequality2}
&&\int_{B_2}(|\nabla^k u|^2+|\nabla^{k+1} d|^2)\eta^2(T_*)
+\frac12\int_{t*}^{T_*}\int_{B_2}\left(\mu_4|\nabla^{k+1} u|^2+\frac{2}{|\lambda_1|}|\nabla^{k+2} d|^2\right)\eta^2\nonumber\\
&\leq& e^{C\phi(T_*)}\Big[\int_{B_{2}}(|\nabla^k u|^2+|\nabla^{k+1} d|^2)\eta^2(t_*)+C
\sum_{l=1}^k\int_{P_{1+2^{-(l+1)}}}\left(|\nabla^l u|^2+|\nabla^{l+1}d|^2+|\nabla^{l-1} P|^2\right)\Big]\nonumber\\
&+&C\int_{P_\frac32}\big[|u|^8+|\nabla d|^8\big]
+C\sum_{l=0}^{k-1}\int_{P_{1+2^{-(l+1)}}}\big[|\nabla^l u|^4+|\nabla^{l+1} d|^4+|\nabla^{l+1} (d\#d)|^4
+|\nabla^{l+1}(d\#d\#d)|^4)\big]\nonumber\\
&\leq& e^{C\epsilon_0}(8\epsilon_0+C\epsilon_0),
\eeq
where we have used both (\ref{induction-hypothesis}), (\ref{l4-estimate-ud}), and (3.69) in the last step.

It is easy to see that we can choose
$\displaystyle C_1>(C+8)e^{C\epsilon_0}$
so that
$\displaystyle e^{C\epsilon_0}(8\epsilon_0+C\epsilon_0)<C_1\epsilon_0.$
Hence  (\ref{2nd_order_inequality2}) implies
$$\int_{B_2}(|\nabla^k u|^2+|\nabla^{k+1} d|^2)\eta^2(T_*)<C_1\epsilon_0,$$
which contradicts the definition of $T_*$. Thus the claim holds true.

Since $\nabla^k P$ satisfies
\beq\label{P-equation}
\Delta \nabla^kP=-(\nabla\cdot)^2\left[\nabla^k(u\otimes u+\nabla d \odot\nabla d-\sigma^L(u,d))\right],
\eeq
the elliptic theory and (\ref{2nd_order_inequality2}) (with $T_*=0$) then yield
\beq\label{P-estimate}
\int_{P_{1+2^{-(k+2)}}}|\nabla^k P|^2&\lesssim&\sup_{-(1+2^{-(k+1)})^2\le t\le 0}\int_{B_{1+2^{-(k+1)}}}(|\nabla^k u|^2+|\nabla^{k+1} d|^2)
\nonumber\\
&&+\int_{P_{1+2^{-(k+1)}}}\left(|P|^2+|\nabla^{k+1} u|^2+|\nabla^{k+2} d|^2\right)
\leq C\epsilon_0.
\eeq
This yields that the conclusion holds for $l=k$. Thus the proof is complete.
\qed

\bigskip
\noindent{\bf Completion of Proof of Theorem \ref{regularity}}:
It is readily seen that by the Sobolev embedding theorem, lemma \ref{higher-order-estimate} implies that
$(\nabla^k u,\nabla^{k+1} d)\in L^\infty(P_{\frac{3r_0}4}(z_0))$ for any $k\ge 1$. This, combined with the theory of linear Stokes' equation and heat equation, would
imply the smoothness of $(u,d)$ in $P_{\frac{r_0}2}(z_0)$.
\qed

\setcounter{section}{3} \setcounter{equation}{3}
\section{Existence of global weak solutions of Ericksen-Leslie's system  (\ref{ES2})}

In this section, by utilizing both the local energy inequality (\ref{local_energy_ineq1}) for suitable weak solutions of (\ref{ES2})
and the regularity Theorem \ref{regularity} for suitable weak solutions to (\ref{ES2}),
we will establish the existence of global weak solutions to (\ref{ES2}) and (\ref{IV})
that enjoy the regularity described as in Theorem \ref{existence}. The argument is
similar to \cite{LLW} Section 5.

First, we recall the following version of Ladyzhenskaya's inequality (see \cite{struwe} lemma 3.1 for the proof).
\begin{lemma}\label{le:4.2} There exists $C_0>0$
such that for any $T>0$, if  $u\in L^\infty_t L^{2}_x\cap L^2_tH^1_x(\R^2\times [0, T])$,
then for $0<R\le+\infty$, it holds
\beq \label{lady1} \int_{\R^2\times [0,T]}| u|^4 \leq C_0
\big(\sup\limits_{(x,t)\in \R^2\times [0,T]}\int_{B_R(x)}|u|^2(\cdot, t)\big)
\big[\int_{\R^2\times [0,T]}|\nabla u|^2+\frac{1}{R^2}\int_{\R^2\times[0,T]}| u|^2\big].
\eeq
\end{lemma}
\bigskip

Similar to \cite{LLW} lemma 5.2, we  can estimate the life span of smooth solutions  to (\ref{ES2})
in term of Sobolev profiles of smooth initial data.

\begin{lemma} \label{life_est} Assume (\ref{necessary}), (\ref{parodi}), and (\ref{Leslie_condition}) hold.
There exist $\epsilon_1>0$  and $\theta_1>0$ depending on $(u_0,d_0)$
such that if $(u_0,d_0)\in C^\infty(\R^2, \mathbb R^2\times \mathbb S^2)\bigcap_{k\ge 0} \big(H^k(\R^2,\R^2)\times H^{k+1}_{e_0}(\R^2,\mathbb S^2)\big)$
satisfies
\beq \label{small_data}
\sup_{x\in\R^2}\int_{B_{2R_0}(x)}(|u_0|^2+|\nabla d_0|^2)
\le\epsilon_1^2
\eeq
for some $R_0>0$.
Then there exist $T_0\ge\theta_1 R_0^2$ and a unique solution
$(u,d)\in C^\infty(\R^2\times [0,T_0],\mathbb R^2\times \mathbb S^2)$ of (\ref{ES2}) and (\ref{IV}) in $\R^2$, satisfying
\beq
\label{small_data10}
\sup_{(x,t)\in\R^2\times [0, T_0]}\int_{B_{R_0}(x)}(|u|^2+|\nabla d|^2)(t)\le 2\epsilon_1^2.
\eeq
\end{lemma}

\noindent{\it Proof.}  By the theorem of Wang-Zhang-Zhang \cite{WZZ} on the local existence of smooth solutions,   there exist  $T_0>0$ and
a unique smooth solution $(u,d)\in C^\infty(\R^2\times [0,T_0],\mathbb R^2\times \mathbb S^2)$ to (\ref{ES2}) and (\ref{IV}).
Let $0< t_0\le T_0$ be the maximal time such that
\beq\label{small_data1}
\sup_{x\in\R^2}\int_{B_{R_0}(x)}(|u|^2+|\nabla d|^2)(t)\le 2\epsilon_1^2,  \ 0\le t\le t_0.
\eeq
Hence we have
\beq\label{small_data2}
\sup_{x\in\R^2}\int_{B_{R_0}(x)}(|u|^2+|\nabla d|^2)(t_0)=2\epsilon_1^2.
\eeq
Without loss of generality, we assume $t_0\le R_0^2$.
Set
$$ E(t)=\int_{\R^2}(| u|^2+|\nabla d|^2)(t), \ \
E_0=\int_{\R^2}(| u_0|^2+|\nabla d_0|^2).$$
Then by lemma \ref{global_energy_ineq} we have that for any $0<t\le t_0$,
\beq\label{global_energy9} E(t)+\int_{\R^2\times [0, t]}\left(\mu_4|\nabla u|^2+\frac{2}{|\lambda_1|}\left|\triangle d+|\nabla d|^2 d\right|^2\right)
\leq E_0.
\eeq
By lemma \ref{le:4.2}, we have that for all $0\le t\le t_0$,
\beq
\int_{\R^2\times [0, t]}|\nabla d|^4&\leq& C_0
\left(\sup\limits_{(x,s)\in \R^2\times [0, t]}\int_{B_{R_0}(x)}
|\nabla d|^2(s)\right)
\left[\int_{\R^2\times [0,t]}|\triangle d|^2 +\frac{1}{R_0^2}\int_{\R^2\times [0, t]}|\nabla d|^2\right]\nonumber\\
&\leq& C_0\mathcal{E}_{R_0}^2(t)\Big[\int_{\R^2\times [0,t]}|\triangle
d|^2+\frac{tE_0}{R_0^2}\Big]
\label{lady20}
\eeq where
$$\mathcal{E}_{R_0}^1(t)=\sup\limits_{(x,s)\in \R^2\times [0,t]}\int_{B_{R_0}(x)}| u|^2(s),
\ \mathcal{E}_{R_0}^2(t)=\sup\limits_{(x,s)\in \R^2\times [0, t]}\int_{B_{R_0}(x)}|\nabla d|^2(s),
 \ \mathcal{E}_{R_0}(t)=\sum_{i=1}^2\mathcal{E}_{R_0}^i(t).$$
By (\ref{small_data1}), we have $\displaystyle\mathcal{E}_{R_0}(t)\le 2\epsilon_1^2, \ \forall \ 0\le t\le t_0$
so that
\beq \label{lady3}\int_{\R^2\times [0, t_0]}|\nabla d|^4\leq
C_0\epsilon_1^2\Big[\int_{\R^2\times [0, t_0]}|\triangle{d}|^2+\frac{t_0 E_0}{R_0^2}\Big].
\eeq
Hence we obtain
\beqno
\int_{\R^2\times [0, t_0]}|\triangle{d}|^2&=& \int_{\R^2\times [0, t_0]}\left(|\triangle d+|\nabla d|^2 d|^2+|\nabla
d|^4\right)\leq \frac{|\lambda_1|}2 E_0+\int_{\R^2\times [0, t_0]}|\nabla d|^4\\
&\leq& C_0\epsilon_1^2\int_{\R^2\times [0, t_0]}|\triangle{d}|^2+\left(\frac{C_0t_0\epsilon_1^2}{R_0^2}+\frac{|\lambda_1|}2\right) E_0.
\eeqno
Choosing $\displaystyle 0<\epsilon_1^2\le \frac{1}{2C_0}$, we would have
\beq\label{22-est-d}
\int_{\R^2\times [0, t_0]}|\nabla^2 d|^2=\int_{\R^2\times [0, t_0]}|\triangle d|^2\leq \left(|\lambda_1|+\frac{C_0\epsilon_1^2t_0}{R_0^2}\right)E_0
\le C_1E_0.
\eeq
This, combined with (\ref{lady3}), also yields
\beq\label{l4-est-Dd}
\int_{\R^2\times [0, t_0]}|\nabla d|^4\le C_1\epsilon_1^2E_0.
\eeq
We can also estimate
\beq \int_{\R^2\times [0, t_0]}| u|^4 &\leq& C_0\mathcal{E}_{R_0}^1(t_0)
\Big[\int_{\R^2\times [0, t_0]}|\nabla u|^2+\frac{1}{R_0^2}\int_{\R^2\times [0, t_0]}| u|^2\Big]\nonumber\\
&\leq& C_0\mathcal{E}_{R_0}^1(t_0)\left(\frac{E_0}{\mu_4}+\frac{t_0E_0}{R_0^2}\right)\leq C_1\epsilon_1^2 E_0.
\label{l4-est-u}
\eeq

Now we need to estimate $\mathcal{E}_{R_0}(t)$. Before we do it, we need to recall the following global $L^2$-estimate of $P$:
\beq\label{global_P_estimate}
\int_0^{t_0}\int_{\R^2} |P|^2\lesssim \int_0^{t_0}\int_{\R^2}(|u|^4+|\nabla d|^4+|\sigma^L(u,d)|^2)
\lesssim \int_0^{t_0}\int_{\R^2}(|u|^4+|\nabla d|^4+|\nabla u|^2+|\nabla^2 d|^2).
\eeq
For any $x\in\R^2$,
let $\eta\in C^\infty_0(B_{2R_0}(x))$ be a cut-off function of $B_{R_0}(x)$:
$$0\le\eta\le 1, \ \eta\equiv 1 \ {\rm{on}}\  B_{R_0}(x), \ \eta\equiv 0
\ {\rm{outside}}\  B_{2R_0}(x), \ |\nabla\eta|\leq 4R_0^{-1}.$$
Then, by applying lemma \ref{local_energy_ineq} with this $\eta$ and the estimates (\ref{global_energy9}), (\ref{l4-est-Dd}), (\ref{22-est-d}),  (\ref{l4-est-u}),
and (\ref{global_P_estimate}),
we have
\beq&&\sup\limits_{0\leq t\leq t_0}\int_{B_{R_0}(x)}(| u|^2+|\nabla d|^2)+c_0\int_0^{t_0}\int_{B_{R_0}(x)}(|\nabla u|^2+|\Delta d+|\nabla d|^2 d|^2)
-\mathcal{E}_{2R_0}(0)\nonumber\\
&\lesssim &\int_{\R^2\times [0, t_0]}
\Big[| u|^3+|u||\nabla u|+|u|^2|\nabla d|+|u||\nabla^2 d|+|u||P|+|\nabla d||\nabla^2 d|+|u||\nabla d|^2+|\nabla d||\nabla u|\Big]
|\nabla\eta^2|\nonumber\\
&\lesssim& \frac{1}{R_0}\Big\||u|+|\nabla d|\Big\|_{L^2(\R^2\times [0,t_0])}
\Big[\left\||\nabla u|+|\nabla^2 d|+|P|\right\|_{L^2(\R^2\times [0,t_0])}+\left\||u|+|\nabla d|\right\|_{L^4(\R^2\times [0,t_0])}^2
\Big]\nonumber\\
&\leq& C\left(\frac{t_0}{R_0^2}\right)^\frac12 E_0^\frac12.
\label{sup-est}\eeq
 Thus, by taking supremum over $x\in\R^2$ we obtain
\beq\label{sup-est1}
2\epsilon_1^2&=&\sup\limits_{x\in\R^2, \ 0\leq t\leq t_0}\int_{B_{R_0}(x)}(| u|^2+|\nabla d|^2)(t)\nonumber\\
&\leq & \mathcal{E}_{2R_0}(0)+ C\left(\frac{t_0}{R_0^2}\right)^\frac12 E_0^\frac12
\leq  \epsilon_1^2+ C_0\left(\frac{t_0}{R_0^2}\right)^{\frac12}E_0^\frac12.
\eeq
This implies
$$t_0\ge \frac{\epsilon_1^2}{C_0^2 E_0} R_0^2=\theta_1 R_0^2,
\ \  {\rm{with}}\ \ \theta_1\equiv \frac{\epsilon_1^2}{C_0^2 E_0}.$$
Set $T_0=t_0$. Then we have that $T_0\ge \theta_1R_0^2$ and (\ref{small_data10}) holds. This
completes the proof.
\qed

\bigskip
Before proving Theorem \ref{existence}, we need the following density property of Sobolev maps (see \cite{schoen-uhlenbeck}
for the proof).

\begin{lemma}\label{density} For $n=2$ and any given map $f\in H^1_{e_0}(\R^2,\mathbb S^2)$,
there exist $\{f_k\}\subset C^{\infty}(\R^2,\mathbb S^2)\bigcap_{l\ge 1} H^l_{e_0}(\R^2,\mathbb S^2)$ such that
$$\lim_{k\rightarrow\infty}\|f_k-f\|_{H^1(\R^2)}=0.$$
\end{lemma}

\bigskip
\noindent{\bf Proof of Theorem \ref{existence}}:

Since $u_0\in \mathbf H$, there exists $u_0^k\in C^\infty_0(\R^2,\mathbb R^2)$, with $\nabla\cdot u_0^k=0$,
such that
$$\lim_{k\rightarrow\infty}\|u_0^k-u_0\|_{L^2(\R^2)}=0.$$
Since $d_0\in H^1_{e_0}(\R^2,\mathbb S^2)$, lemma \ref{density}
implies that there exist
$\{d_0^k\}\subset C^{\infty}(\R^2,\mathbb S^2)\bigcap_{l\ge 1} H^l_{e_0}(\R^2,\mathbb S^2)$ such that
$$\lim_{k\rightarrow\infty}\|d_0^k-d_0\|_{H^1(\R^2)}=0.$$
By the absolute continuity of $\displaystyle\int (|u_0|^2+|\nabla d_0|^2)$,
there exists $R_0>0$ such that
\beq \label{small_data11}
\sup_{x\in\R^2}\int_{B_{2R_0}(x)}(|u_0|^2+|\nabla d_0|^2)
\le\frac{\epsilon_1^2}{2},
\eeq
where $\epsilon_1>0$ is given by lemma \ref{life_est}. By the strong convergence
of $(u_0^k,\nabla d_0^k)$ to $(u_0,\nabla d_0)$ in $L^2(\R^2)$, we have that
\beq \label{small_data12}
\sup_{x\in\R^2}\int_{B_{2R_0}(x)}(|u_0^k|^2+|\nabla d_0^k|^2)
\le {\epsilon_1^2}, \ \forall k>>1.
\eeq
For simplicity, we assume (\ref{small_data12})
holds for all $k\ge 1$.  By lemma \ref{life_est},
there exist $\theta_0=\theta_0(\epsilon_1, E_0)\in (0, 1)$
and $T_0^k\ge \theta_0 R_0^2$ such that there exist
solutions
$(u^k,d^k)\subset C^\infty(\R^2\times [0, T_0^k],\mathbb R^2\times \mathbb S^2)$  to (\ref{ES2}) and (\ref{IV}) with the initial condition:
\beq\label{initial_cond}
(u^k,d^k)\big|_{t=0}=(u_0^k,d_0^k),
\eeq
that satisfies
\beq\label{small_data13}
\sup_{(x,t)\in\R^2\times [0, T_0^k]}\int_{B_{R_0}(x)}(|u^k|^2+|\nabla d^k|^2)(t)
\le 2\epsilon_1^2, \ \forall k\ge 1.
\eeq
By lemma \ref{global_energy_ineq}, we have that for all $k\ge 1$,
\beq\label{global_energy20}
&&\sup_{0\le t\le T_0^k}\int_{\R^2} (|u^k|^2+|\nabla d^k|^2)(t)
+\int_{\R^2\times [0, T_0^k]}(\mu_4|\nabla u^k|^2+\frac{2}{|\lambda_1|}|\nabla^2 d^k|^2)\nonumber\\
&&\le \int_{\R^2} (|u_0^k|^2+|\nabla d_0^k|^2)\le CE_0.
\eeq
Combining (\ref{small_data13}) and (\ref{global_energy20})  with
lemma \ref{life_est}, we conclude that
\beq \label{l4-est-uk-dk}
\int_{\R^2\times [0, T_0^k]}(|u^k|^4+|\nabla d^k|^4)
\le C\epsilon_1^2 E_0,
\eeq
\beq\label{22-est-dk}
\int_{\R^2\times [0, T_0^k]}(|\partial_t d^k|^2+|\nabla^2 d^k|^2+|P^k|^2)\le CE_0,
\eeq
and
\beq\label{22-est-dk1}
\sup_{x\in\R^2}\int_0^{T_0^k}\int_{B_{R_0}(x)}(|\nabla u^k|^2+|\nabla^2 d^k|^2+|P^k|^2)\le C\epsilon_1^2.
\eeq
Furthermore,  (\ref{ES2}) implies that for any $\phi\in \mathbf J$,
$$
\langle \partial_t u^k, \phi\rangle=-\int_{R^2} \nabla u^k\cdot\nabla \phi
+\int_{\R^2} (u^k\otimes u^k +\nabla d^k\odot\nabla d^k-\sigma^L(u^k, d^k)):\nabla\phi,
$$
where $\langle\cdot,\cdot\rangle$ denotes the pair between $H^{-1}(\R^2)$ and $H^1(\R^2)$,
we conclude that $\partial_t u^k\in L^2([0,T_0^k], H^{-1}(\R^2))$ and
\beq\label{weak-est-ut}
\left\|\partial_t u^k\right\|_{L^2([0,T_0^k], H^{-1}(\R^2))}\le C E_0.
\eeq
It follows from (\ref{l4-est-uk-dk}) and (\ref{22-est-dk1}) that
$$\Phi\left(u^k, d^k, P^k, (x,t), R_0\right)\le C\epsilon_1^2, \ \forall x\in\mathbb R^2, R_0^2\le t\le T_0^k.$$
Hence, by Theorem \ref{regularity}, we conclude that for any $\delta>0$,
\beq\label{c2-est}
\left\|(u^k, d^k)\right\|_{C^{l}(\R^2\times [\delta, T_0^k])}\leq C\left(l, \epsilon_1, E_0, \delta\right), \ \forall l\ge 1.
\eeq
After passing to a subsequence, we may assume that there exist $T_0\ge \theta_0R_0^2$,
$u\in L^\infty_tL^2_x\cap L^2_tH^1_x(\R^2\times [0, T_0],\mathbb R^2)$,
$d\in L^\infty_tH^1_{e_0}\cap L^2_tH^2_{e_0}(\R^2\times [0, T_0], \mathbb S^2)$,
and $P\in L^2(\R^2\times [0, T_0])$ such that
$$u^k\rightharpoonup u \ {\rm{in}}\ L^2_tH^1_x(\R^2\times [0, T_0], \mathbb R^2),\
d^k\rightharpoonup d \ {\rm{in}}\ L^2_t H^2_{e_0}(\R^2\times [0, T_0], \mathbb S^2),
\ P^k\rightharpoonup P \ {\rm{in}}\ L^2(\R^2\times [0,T_0]).$$
It follows from (\ref{22-est-dk}) and (\ref{weak-est-ut}) that we can apply Aubin-Lions' lemma to conclude
that
$$u^k\rightarrow u, \ \nabla d^k\rightarrow \nabla d \ \ {\rm{in}}\ \  L^4_{\rm{loc}}(\R^2\times [0, T_0]).
$$
By (\ref{c2-est}), we may assume that for any $0<\delta<T_0, 0<R<+\infty$ and $l\ge 1$,
$$\lim_{k\rightarrow\infty}\left\|(u^k,d^k)-(u,d)\right\|_{C^l(B_R\times [\delta, T_0])}=0.$$
It is clear that $u\in C^{\infty}(\R^2\times (0,T_0],\mathbb R^2)\cap
(L^\infty_tL^2_x\cap L^2_tH^1_x)(\R^2\times [0,T_0],\R^2)$,
$d\in C^\infty(\R^2\times (0, T_0], \mathbb S^2)\cap  (L^2_tH^1_{e_0}\cap L^2_tH^2_{e_0})(\R^2\times [0,T_0],\mathbb S^2)$ solves
(\ref{ES2}) in $\R^2\times (0,T_0]$.
It follows from (\ref{weak-est-ut}) and (\ref{22-est-dk})
that we can assume
$$(u,\nabla d)(t)\rightharpoonup (u_0,\nabla d_0)\  {\rm{in }}\ L^2(\R^2), \ {\rm{as}}\ t\downarrow 0.$$
In particular, by the lower semicontinuity we have that
$$E(0)\le \liminf_{t\downarrow 0}E(t).$$
On the other hand, (\ref{global_energy20}) implies
$$E(0)\ge \limsup_{t\downarrow 0}E(t).$$
This implies that $(u,\nabla d)(t)\rightarrow (u_0,\nabla d_0)$ in $L^2(\R^2)$ and hence $(u,d)$ satisfies the initial condition (\ref{IV}).

Let $T_1\in [T_0, +\infty)$ be the first finite singular time of $(u,d)$, i.e.,
$$(u,d)\in  C^{\infty}(\R^2\times (0,T_1),\mathbb R^2\times \mathbb S^2)
\ \ {\rm{but}}\ \ (u,d)\notin  C^{\infty}(\R^2\times (0,T_1],\mathbb R^2\times \mathbb S^2).$$
Then  we must have
\beq\label{char_1st_singu}
\limsup_{t\uparrow T_1}\max_{x\in\R^2}\int_{B_R(x)}
(|u|^2+|\nabla d|^2)(t)\ge \epsilon_1^2,  \ \forall R>0.
\eeq
Now we look for an extension of this weak solution beyond $T_1$.  To do it,
we define the new initial data at $t=T_1$.  \\
\noindent{\it Claim 2}. $(u,d)\in C^0([0,T_1], L^2(\R^2))$. \\This follows easily from (\ref{weak-est-ut}) and (\ref{22-est-dk}).
By Claim 2, we can define
$$(u(T_1), d(T_1))=\lim_{t\uparrow T_1}(u(t),d(t)) \ {\rm{in}}\ L^2(\R^2).$$
By lemma \ref{global_energy_ineq} we have that $\nabla d\in L^\infty([0,T_1], L^2(\R^2))$ so that
$\nabla d(t)\rightharpoonup \nabla d(T_1)$ in $L^2(\R^2)$.
In particular, $u(T_1)\in \mathbf H$ and $d(T_1)\in H^1_{e_0}(\R^2,\mathbb S^2)$.

Use $(u(T_1),d(T_1))$ as an initial data  to obtain a continuation of $(u,d)$ beyond $T_1$ as a weak solution of (\ref{ES2}) and (\ref{IV}),
we will show that this procedure will cease in finite steps and  afterwards we will have constructed a  global weak solution.
In fact, at any such singular time there is at least a loss of energy amount of $\epsilon_1^2$.
By (\ref{char_1st_singu}), there exist $t_i\uparrow T_1$ and $x_0\in\R^2$ such that
$$\limsup_{t_i\uparrow T_1}\int_{B_R(x_0)}(|u|^2+|\nabla d|^2)(t_i)\ge \epsilon_1^2
\ \ {\rm{for\ all}}\ R>0.$$
This implies
\beqno
&&\int_{\R^2} (|u|^2+|\nabla d|^2)(T_1)=\lim_{R\downarrow 0}\int_{\R^2\setminus B_R(x_0)} (|u|^2+|\nabla d|^2)(T_1)
\le \lim_{R\downarrow 0}\liminf_{t_i\uparrow T_1}\int_{\R^2\setminus B_R(x_0)} (|u|^2+|\nabla d|^2)(t_i)\nonumber\\
&&\le \lim_{R\downarrow 0}\ \Big[\liminf_{t_i\uparrow T_1}\int_{\R^2}(|u|^2+|\nabla d|^2)(t_i)
-\limsup_{t_i\uparrow T_1}\int_{B_R(x_0)}(|u|^2+|\nabla d|^2)(t_i)\Big]
\le  E_0-\epsilon_1^2.
\eeqno
Hence the number of finite singular times must be bounded by $\displaystyle L=\left[\frac{E_0}{\epsilon_1^2}\right]$.
If $0<T_L<+\infty$ is the largest finite singular time, then we can use $(u(T_L), d(T_L))$
as the initial data to construct a weak solution $(u,d)$
to (\ref{ES2}) and (\ref{IV})  in $\R^2\times [T_L, +\infty)$.  Thus (i) of Theorem \ref{existence}
is established.  It is also clear that (iii) of Theorem \ref{existence} holds for the solution constructed.

Now, we perform a blow-up analysis at each finite singular time.
It follows from (\ref{char_1st_singu}) that there exist $0<t_0<T_1$, $t_m\uparrow T_1$, $r_m\downarrow 0$
such that
\beq\label{concentration1}
\epsilon_1^2
=\sup_{x\in\R^2, t_0\le t\le t_m} \int_{B_{r_m}(x)}(|u|^2+|\nabla d|^2)(t).
\eeq
By lemma \ref{life_est}, there exist $\theta_0$, depending only on $\epsilon_1$ and $E_0$,
and $x_m\in\R^2$
such that
\begin{eqnarray}\label{concentration2}\int_{B_{2r_m}(x_m)}(|u|^2+|\nabla d|^2)(t_m-\theta_0 r_m^2)
\ge\frac12
\sup_{x\in\R^2}\int_{B_{2r_m}(x)}(|u|^2+|\nabla d|^2)(t_m-\theta_0r_m^2)
\ge \frac{\epsilon_1^2}{4}.
\end{eqnarray}
By lemma \ref{life_est} and (\ref{concentration1}), we  have
\beq\label{l4-boundness}
\int_{\R^2\times [t_0,t_m]}(|u|^4+|\nabla d|^4)\le C\epsilon_1^2.
\eeq
Define the blow-up sequence $(u_m, d_m, P_m):\R^2\times
[\frac{t_0-t_m}{r_m^2}, 0]\to \mathbb R^2\times\mathbb S^2\times\R$ by
$$
u_m(x,t)=r_m u(x_m+r_m x, t_m +r_m^2 t), \
d_m(x,t)=d(x_m+r_m x, t_m+r_m^2 t), \ P_m=r_m^2P(x_m+r_m x, t_m +r_m^2 t).$$
Then  $(u_m,d_m, P_m)$ solves (\ref{ES2}) in $\R^2\times [\frac{t_0-t_m}{r_m^2}, 0]$.
Moreover,
$$\int_{B_2(0)}\big(|u_m|^2+|\nabla d_m|^2\big)(-\theta_0)\ge \frac{\epsilon_1^2}{2},$$
$$\int_{B_1(x)}\big(|u_m|^2+|\nabla d_m|^2\big)(t)\le \epsilon_1^2,
\ \forall x\in \R^2, \ \frac{t_0-t_m}{r_m^2}\le t\le 0,$$
$$\int_{\R^2\times [\frac{t_0-t_m}{r_m^2}, 0]}\big(|u_m|^4+|\nabla d_m|^4\big)\le C\epsilon_1^2,$$
and
$$\int_{P_2(z)}\left(|\nabla u_m|^2+|\nabla^2 d_m|^2+|P_m|^2\right)\le C\epsilon_1^2,
\ \forall\ z=(x,t)\in \R^2\times [\frac{t_0-t_m}{r_m^2}, 0].$$
It is easy to see $\frac{t_0-t_m}{r_m^2}\rightarrow -\infty$.
Hence, by Theorem \ref{regularity}, we can assume that there exists a smooth solution
$(u_\infty,d_\infty):\mathbb R^2\times (-\infty, 0]\to \mathbb R^2\times \mathbb S^2$ to (\ref{ES2}) such
that for any $l\ge 1$,
$$(u_m,d_m)\rightarrow (u_\infty,d_\infty) \ {\rm{in}} \ C^l_{\rm{loc}}(\mathbb R^2\times [-\infty,0]).$$
\noindent{\it Claim 3}. $u_\infty\equiv 0$. \\In fact,
since $u\in L^4(\R^2\times [0,T_1])$, we have
$$
\int_{P_R}|u_\infty|^4=\lim_{m\rightarrow\infty} \int_{P_R}|u_m|^4
=\lim_{m\rightarrow\infty} \int_{B_{Rr_m}(x_m)}\int_ {t_m-R^2r_m^2}^ {t_m}|u|^4=0.$$
\noindent{\it Claim 4}. {\it $d_\infty\in C^\infty(\R^2,\mathbb S^2)$ is a nontrivial harmonic map with finite energy}.\\
Since $(\Delta d+|\nabla d|^2 d)\in L^2(\R^2\times [0,T_1])$, we have, for any compact $K\subset\mathbb R^2$,
\begin{eqnarray*}
\int_{-2\theta_0}^0\int_{K}|\Delta d_\infty+|\nabla d_\infty|^2 d_\infty|^2
&\le &\liminf_{m} \int_{-2\theta_0}^0\int_{\Omega_m}|\Delta d_m+|\nabla d_m|^2 d_m|^2\\
&=& \lim_{m} \int_{t_m-2\theta_0 r_m^2}^{t_m}\int_{\R^2}|\Delta d+|\nabla d|^2 d|^2=0.
\end{eqnarray*}
By the equation (\ref{ES2})$_3$ and $u_\infty\equiv 0$, this implies
$\partial_t d_\infty\equiv 0$ on $\mathbb R^2\times [-2\theta_0,0]$.
Hence $d_\infty(t)\equiv d_\infty\in C^\infty(\mathbb R^2, \mathbb S^2)$ is a harmonic map. Since
$$\int_{B_2}|\nabla d_\infty|^2=\lim_{m}\int_{B_2}(|u_m|^2+|\nabla d_m|^2)(-\theta_0)\ge \frac{\epsilon_1^2}{4},$$
$d_\infty$ is a non-constant map. By the lower semicontinuity,
we have that for any ball $B_R\subset\mathbb R^2$,
$$\int_{B_R}|\nabla d_\infty|^2
\le \liminf_{m}\int_{B_R}|\nabla d_m|^2(-\theta_0)
=\liminf_{m}\int_{B_{r_mR}(x_m)}|\nabla d|^2(t_m-\theta_0r_m^2)\le E_0,$$
and hence $d_\infty$ has finite energy.  It is well-known (\cite{sacks-uhlenbeck} \cite{struwe})
that $d_\infty$ can be lifted to be a non-constant harmonic map from $\mathbb S^2$ to $\mathbb S^2$.
In particular, $d_\infty$ has a non zero degree and
$$\int_{\mathbb R^2}|\nabla d_\infty|^2\ge 8\pi|{\rm{deg}}(d_\infty)|\ge 8\pi.$$
It follows from the above argument that  for any $r>0$,
$$\limsup_{t\uparrow T_1}\max_{x\in\R^2}\int_{B_r(x)}(|u|^2+|\nabla d|^2)(t)
\ge \int_{\mathbb R^2}(|u_\infty|^2+|\nabla d_\infty|^2)\ge 8\pi.$$

To show (iv). By lemma \ref{global_energy_ineq}, there exists $t_k\uparrow +\infty$ such that for $(u_k,d_k)=(u(t_k),
d(t_k))$,
$$\int_{\R^2}(|u_k|^2+|\nabla d_k|^2)\le E_0,\ \ \lim_{k\rightarrow\infty}
\int_{\R^2} (|\nabla u_k|^2+|\Delta d_k+|\nabla d_k|^2 d_k|^2)=0.$$
It is easy to see that $u_k\rightarrow 0$ in $H^1(\R^2)$, and
$\{d_k\}\subset H^1_{e_0}(\R^2,\mathbb S^2)$ is a bounded sequence of approximate harmonic maps, with
tension fields $\tau(d_k)=\Delta d_k+|\nabla d_k|^2 d_k$ converging to zero in $L^2(\R^2)$. By the energy identity result
by Qing \cite{qing} and Lin-Wang \cite{lin-wang}, we can conclude that there exist a harmonic map
$d_\infty \in C^\infty\cap H^1_{e_0}(\R^2,\mathbb S^2)$,
and finitely many points $\{x_i\}_{i=1}^l$, $\{m_i\}_{i=1}^l\subset\mathbb N$, such that
$$|\nabla d_k|^2\,dx\rightharpoonup |\nabla d_\infty|^2\,dx+\sum_{i=1}^l 8\pi m_i \delta_{x_i}$$
as convergence of Radon measures. This yields (iv).

To show (v) under the condition $\int_{\R^2}(|u_0|^2+|\nabla d_0|^2)\le 8\pi$,  we divide the argument into two cases:\\
(a) {\it There exist no finite time singularities.} \\ For, otherwise, (ii) implies
that we can blow up near the first singular time $T_1$ to obtain a nontrivial harmonic map
$\omega\in C^\infty(\mathbb R^2, \mathbb S^2)$ and
$$8\pi\le \int_{\mathbb R^2}|\nabla \omega|^2\le\lim_{t\uparrow T_1}
\int_{\R^2} (|u|^2+|\nabla d|^2)(t)\le \int_{\R^2} (|u_0|^2+|\nabla d_0|^2)\le 8\pi.$$
This, combined with lemma \ref{global_energy_ineq}, yields
$$\int_0^{T_1}\int_{\R^2} (|\nabla u|^2+|\Delta d+|\nabla d|^2 d|^2)=0$$
so that $u=\partial_t d\equiv 0$ in $\R^2\times [0,T_1]$ and hence $d(t)=d_0\in C^\infty\cap H^1_{e_0}(\R^2, \mathbb S^2)$,
$0\le t\le T_1$, is a harmonic map.   This contradicts the fact that $T_1$ is a singular time. \\
(b) {\it $\displaystyle\phi(t)\equiv\sup_{x\in\R^2, \tau\le t}(|u|+|\nabla d|)(x,\tau)$ remains bounded
as $t\uparrow +\infty$}.\\
For, otherwise, there exist $t_k\uparrow +\infty$ and $x_k\in\R^2$ such that
$\displaystyle\lambda_k=\phi(t_k)=(|u|+|\nabla d|)(x_k,t_k)\rightarrow +\infty.$
Define  $(u_k,d_k):\R^2\times [-t_k\lambda_k^2,0]
\to \mathbb R^2\times \mathbb S^2$ by
$$ u_k(x,t)=\lambda_k^{-1}u(x_k+\lambda_k^{-1}x, t_k+\lambda_k^{-2} t),
\ d_k(x,t)=d(x_k+\lambda_k^{-1}x, t_k+\lambda_k^{-2} t), P_k(x,t)=\lambda_k^{-2} P(x_k+\lambda_k^{-1}x,
t_k+\lambda_k^{-2} t).$$
Then $(u_k,d_k, P_k)$ solves (\ref{ES2}) on $\R^2\times [-t_k\lambda_k^2,0]$, and
$$1=(|u_k|+|\nabla d_k|)(0,0)
\ge (|u_k|+|\nabla d_k|)(x,t), \ \forall (x,t)\in \R^2\times [-t_k\lambda_k^2,0].$$
As in the proof of (ii), we can conclude that
$(u_k,d_k)\rightarrow (0, d_\infty)$ in $C^2_{\rm{loc}}(\mathbb R^2)$,
where $d_\infty\in C^\infty(\mathbb R^2, \mathbb S^2)$ is a nontrivial harmonic map with finite energy.
As in (a), this implies that
$$\int_0^\infty\int_{\R^2}(|\nabla u|^2+|\Delta d+|\nabla d|^2 d|^2)=0$$
so that $u=\partial_t d\equiv 0$ in $\R^2\times [0,+\infty)$ and hence $d(t)=d_0\in C^\infty(\R^2,\mathbb S^2)$,
$0\le t<+\infty$, is a harmonic map. This implies that $\phi(t)$ is constant for $0<t<+\infty$ and
we get a desired contradiction.

Since $\phi(t)$ is a bounded function of $t\in (0,+\infty)$, the higher order regularity Theorem
\ref{regularity} implies that
$$\|\nabla^l u(t)\|_{C^0(\R^2)}+\|\nabla^{l+1} d(t)\|_{C^0(\R^2)}\le C(l), \ \forall\ l\ge 1,  \ \forall\  t\ge 1. $$
Thus we can choose $t_k\uparrow \infty$ such
that
$$\int_{\R^2} \left(|u|^2+|\nabla u|^2\right)(t_k)\le E_0,
\ \ \int_{\R^2} \left(|\nabla u|^2+|\Delta d+|\nabla d|^2d|^2\right)(t_k)\rightarrow 0,$$
and
$$\|u(t_k)\|_{C^2(\R^2)}+\|d(t_k)\|_{C^2(\R^2)}\le C.$$
Thus we may assume that there exist a harmonic map $d_\infty\in C^\infty\cap H^1_{e_0}(\R^2, \mathbb S^2)$ such that
$$(u(t_k), d(t_k))\rightarrow (0, d_\infty)
\ {\rm{in}}\ C^2_{\rm{loc}}(\R^2, \mathbb S^2).$$
This proves (v) under the first condition.

To show (v) under the condition $(d_0)_3\ge 0$. We argue as follows. First, we can approximate $(u_0, d_0)$ by smooth initial data
$(u_0^k, d_0^k)$ such that the third component of $d_0^k$ is non-negative, i.e.,  $(d_0^k)_3\ge 0$. Then we can check that
the short time smooth solutions $(u^k, d^k)$ of (\ref{ES2}),  with the initial data $(u_0^k, d_0^k)$ at $t=0$, on $\R^2\times [0,T_k]$,
have bounded gradients:
$$\Big\||\nabla u^k|+|\nabla d^k|\Big\|_{C^0(\R^2\times [0, T_k])}\le C(k)<+\infty.$$
Since $(d^k)_3$ satisfies the equation
$$\big(\partial_t-\frac{1}{|\lambda_1|}\Delta\big) (d^k)_3+u\cdot\nabla (d^k)_3=\big(\frac{1}{|\lambda_1|}|\nabla d^k|^2+
\frac{\lambda_2}{\lambda_1}\big((\hat {d^k})^TA_k \hat {d^k}\big) (d^k)_3,$$
and the coefficient in front of $(d^k)_3$, $\big(\frac{1}{|\lambda_1|}|\nabla d_k|^2
+\frac{\lambda_2}{\lambda_1}((\hat {d^k})^TA_k \hat {d^k}\big)$ is bounded. Hence we can apply the maximum
principle (see \cite{LSU}) to conclude that $(d^k)_3\ge 0$ in $\R^2\times [0,T_k]$. Sending $k$ to infinity, we conclude that
the global weak solution $(u,d)$  to (\ref{ES2}) and (\ref{IV}), obtained in the part (i), satisfies $d_3\ge 0$. If $(u,d)$ has any finite time singularity,
then by performing the blow-up argument we would obtain a nontrivial harmonic map $\omega$ from $\mathbb S^2$ to $\mathbb S^2$ such that
$\omega_3\ge 0$, which is impossible. Hence $(u,d)$ has no finite time singularity. If $(u,d)$ has singularity at the time infinity, then we would
also obtain a nontrivial harmonic map from $\mathbb S^2$ to the upper hemisphere, which is also impossible. Therefore, $(u,d)$ has bounded
$C^2$-norm in $\R^2\times (\delta,+\infty)$ for any $\delta>0$. This proves (v) under the second condition. 
The proof of Theorem \ref{existence} is now complete.
\qed
\bigskip

\noindent{\bf Added Note}. The third author presented the main results of this article in the workshop 
``Nonlinear analysis of continuum theories: statics and dynamics'' at the University of Oxford,
April 8-12, 2013. During the finalization of this paper, Wendong Wang sent the third author his preprint 
``{GLOBAL EXISTENCE OF WEAK SOLUTION FOR THE 2-D ERICKSEN-LESLIE SYSTEM}'',
in which they also claimed an existence result, similar to part (i) of our Theorem 1.4. However, since their proof is based on
a global energy inequality of second order before the first time of energy concentration,  it fails to be complete. 

\bigskip
\section*{Acknowledgment} The second author is partially supported by NSF grants. The third author is partially supported by NSF grants 1001115 and 1265574,
NSF of China grant 11128102, and a Simons Fellowship in Mathematics. Part of this work was done during the first author's visit to the third author
at the University of Kentucky. The first author would like to thank the Department of Mathematics, University of Kentucky for its hospitalities.

\vskip 0.8cm


\end{document}